\documentclass[11pt]{amsart}
\usepackage[foot]{amsaddr}
\usepackage{color}
 
\usepackage{amsmath,amsfonts}
\usepackage{amsfonts,amsmath,amsthm,amssymb,latexsym}
\usepackage{tikz}
\usepackage{mathrsfs}
\usepackage{mathtools}
\usepackage{algorithm}
\usepackage{algorithmic}
\usepackage{ulem}
\usepackage{hyperref}
\usepackage[toc,page]{appendix}
\usepackage{MnSymbol}
\usepackage{url}

\newtheorem{theorem}{Theorem}[section]
\newtheorem{proposition}[theorem]{Proposition}
\newtheorem{lemma}[theorem]{Lemma}
\newtheorem{corollary}[theorem]{Corollary}
\theoremstyle{definition}
\newtheorem{definition}[theorem]{Definition}
\newtheorem{example}[theorem]{Example}

\theoremstyle{remark}
\newtheorem{remark}[theorem]{Remark}
\numberwithin{equation}{section}
\theoremstyle{remark}

\DeclareMathOperator{\lcm}{lcm}
\DeclareMathOperator{\num}{num}
\DeclareMathOperator{\denom}{denom}

\newcommand{\C}{\mathbb{C}}
\newcommand{\K}{\mathbb{K}}
\newcommand{\LL}{\mathbb{L}}

\newcommand{\FF}{\mathbb{F}}
\newcommand{\Q}{\mathbb{Q}}

\newcommand{\N}{\mathbb{N}}
\newcommand{\V}{\mathbb{V}}
\newcommand{\Pa}{\mathcal{P}}
\newcommand{\cP}{\mathcal{P}}

\newcommand{\Cu}{\mathcal{C}}
\newcommand{\Qa}{\mathcal{Q}}

\newcommand{\Res}{\mathrm{res}}

\def \aa{\mathbf{a}}
\def \aao{\mathbf{a}^0}
\def \S{\mathbb{S}}
\def \res{\mathrm{res}}
\def \No{\mathrm{Norm}}
\def \gg{\mathfrak{g}}
\def \genus{\mathrm{genus}}
\def \sing{\mathrm{sing}}

\def \JJ{\mathrm{J}}
\def \fam{\mathcal{F}}
\def \CuFam{\Cu_{\fam}}
\def \Stand{\mathscr{D}}

\allowdisplaybreaks

\begin{document}

{\color{red} \noindent The journal version of this paper appears in [Sebastian Falkensteiner, J. Rafael Sendra,
Rationality and parametrizations of algebraic curves under specializations,
Journal of Algebra, Volume 659, 2024,
Pages 698-744,
ISSN 0021-8693. 
 { https://doi.org/10.1016/j.jalgebra.2024.07.009} ]
 
 \vspace*{2mm}
 
\noindent licensed under  the CC BY license ({https://creativecommons.org/licenses/by/4.0/}) }

 \vspace*{2mm}

\date{\today}
\title{Rationality and Parametrizations of Algebraic Curves under Specializations}

\author{Sebastian Falkensteiner}
\address{Max Planck Institute Leipzig, Germany.}
\email{sebastian.falkensteiner@mis.mpg.de}

\author{J.Rafael Sendra}
\address{CUNEF Universidad, Department of Quantitative Methods, Spain}
\email{jrafael.sendra@cunef.edu}

\begin{abstract} 
Rational algebraic curves have been intensively studied in the last decades, both from the theoretical and applied point of view. 
In applications (e.g. level curves, linear homotopy deformation, geometric constructions in computer aided design, image detection, algebraic differential equations, etc.), there often appear unknown parameters. It is possible to adjoin these parameters to the coefficient field as transcendental elements. 
In some particular cases, however, the curve has a different behavior than in the generic situation treated in this way. 
In this paper, we show when the singularities and thus the (geometric) genus of the curves might change. 
More precisely, we give a partition of the affine space, where the parameters take values, so that in each subset of the partition the specialized curve is either reducible or its genus is invariant. In particular, we give a Zariski-closed set in the space of parameter values where the genus of the curve under specialization might decrease or the specialized curve gets reducible. 
For the genus zero case, and for a given rational parametrization, a finer partition is possible such that the specialization of the parametrization parametrizes the specialized curve. Moreover, in this case, the set of parameters where Hilbert's irreducibility theorem does not hold can be identified. 
We conclude the paper by illustrating these results by some concrete applications.
\end{abstract}

\maketitle


\section*{Acknowledgements}
 Authors are partially supported by the grant PID2020-113192GB-I00/AEI/10.13039/501100011033 (Mathematical Visualization: Foundations, Algorithms and Applications) from the Spanish State Research Agency (Ministerio de Ciencia, Innovación y Universidades) .
Part of this work was developed during a research visit of the first author to CUNEF University in Madrid.

\section{Introduction}
 
The study and analysis of the behavior of algebraic or algebraic-geometric objects under specializations is of great interest from a theoretical, computational or applied point of view. For instance, some techniques for computing resultants, gcds, or polynomial factorizations, rely on Hensel's lemma or the Chinese remainder theorem (see e.g.~\cite{Geddes}, \cite{winkler}). From a more theoretical point of view, also computational, it is important to control, for instance, when a resultant, or more generally a Gr\"obner basis with parameters, specializes properly (see e.g.~\cite{cox}, \cite{montes}). 
The question whether a given irreducible polynomial over $\K(a_1,\ldots,a_n)$ remains irreducible when the parameters are replaced by values in a field $\K$ was studied intensively by Hilbert~\cite{hilbert1892} and Serre~\cite{Serre} and is the defining property of ``Hilbertian fields''.
The work of Serre can be seen in a more general context. 
Another clear example is the resolution of problems in matrix analysis. Besides the basic analysis of systems with parameters, this type of situation appears in the computation of generalized functional inverses (see e.g.~\cite{penrose}). 

In relation to the applications derived from the particular results of this paper, the following can be mentioned. Geometric constructions in computer aided design, like offsets, conchoids, cissoids etc., introduce parameters as the distance, the matrix entries of an isometry, or the coordinates of the foci; see e.g.~\cite{arrondo1}, \cite{lu}, \cite{cissoids}, \cite{conchoids}.  
Usually it is interesting to analyze geometric properties of these varieties under the specialization of the parameters; for instance, the genus of offsets is analyzed in~\cite{arrondo2}, and the rationality of the conchoids depending on the foci in~\cite{conchoids}. 
Another possible application is the analysis of the rationality of the level curves of surfaces, where one of the variables is taken as a parameter, or the  linear homotopy deformation of curves. Curve recognition, via Hough transform (see e.g.~\cite{torrente2017}), is another field of potential applicability of the results. 
In this context, a catalog of curves, i.e. a family of curves depending on several parameters, as well as a cloud of points, are given. Then, parameter values providing the best approximation, among the catalog, of the cloud of points, is computed. In this frame one may think on determining rational curve solutions in the catalog. 
Finally, one may mention the rational solutions of functional algebraic equations, or the solution of algebraic differential equations which coefficients depend on parameters (see~\cite{FS}). In Section~\ref{sec-application}, we illustrate some of these applications by means of examples.

In this paper, we study algebraic curves $\Cu(F)$ given as the zero-set of a polynomial
\begin{equation}\label{eq-parametric}
F(x,y) = 0 ~\text{ with }~ F \in \K(a_1,\ldots,a_n)[x,y]
\end{equation}
where $\K$ is a computable field of characteristic zero, $a_1,\ldots,a_n$ are a set of parameters, and $F$ is irreducible over the algebraic closure of the coefficient field $\K(a_1,\ldots,a_n)$. We focus on the problem that for certain values of the parameters $a_1,\ldots,a_n$ the algebraic properties of the resulting curve do not coincide with the generic properties of $\Cu(F)$.
More precisely, we define several Zariski-closed sets in the space of parameter values where non-generic behavior may appear. 
Of particular interest are the singularities, their multiplicities and their character. This leads to a partition of the affine space, where the parameters take values, so that in each subset of the partition the specialized curve is either reducible or its (geometric) genus is invariant.
When the generic curve has genus zero, and a rational parametrization, depending on the parameters $a_1,\ldots,a_n$, is given, we provide a new sub-partition where the parametrization specializes properly. 
In particular, the set of parameters where Hilbert's irreducibility theorem does not hold can be identified. Moreover, the birationality of the specialization of the rational parametrization is guaranteed. 
To the best of our knowledge, this is the first systematic computational study of the problem.

In~\cite{walker,Fulton,mauro,myBook}, and references therein, algebraic curves and their rationality are studied. 
The problem of finding rational parametrizations of plane curves is a classical problem and has already been studied by Hilbert and Hurwitz~\cite{HH}, and more recently in~\cite{hoeij}, \cite{schicho}, \cite{SW91}, \cite{SW97}. 
In addition,
for evaluating the parameters, it is important to control field extensions which might be necessary for computing parametrizations. 
Optimal fields of parametrizations have been studied in~\cite{hoeij} and \cite{SW97}.
When introducing parameters in the coefficients, new phenomena have to be considered and lead to Tsen's study of finding solutions in a minimal field~\cite{tsen}.

The structure of the paper is as follows.
In Section~\ref{sec-prel} we present some preliminaries on algebraic curves and rational parametrizations. In Section~\ref{sec-specialization}, we introduce the unspecified parameters and their specialization. 
The computation of the genus and rational parametrizations is followed to define several computable Zariski-open subsets $\Omega$ where the specialized curve behaves, up to irreducibility, as in the generic case. 
The actual computation of the genus is presented in Section~\ref{sec-genus}. 
In Theorem~\ref{theorem:genus-weak} is shown that the genus of the specialized curve, where the parameters take values in $\Omega_{\mathrm{singOrd}}$, is less or equal to the generic genus or the defining polynomial is reducible. 
A direct corollary of that is that specialized curves of rational curves are also rational or reducible (Corollary~\ref{cor:genus-zero}). 
For values in a smaller set $\Omega_{\mathrm{genusOrd}}$, it is shown that the genus of the curve remains exactly the same, again up to irreducibility, see Theorem~\ref{theorem:genus-strong}. 
Section~\ref{sec-parametricCurves} is devoted to the case where the generic curve is rational; in this frame the irreducibility can be guaranteed. 
For some of the parameter values the genus may remain the same but an evaluation of the parametrization is not possible. 
In Theorem~\ref{thrm:genus0general}, however, is presented an open set where the specialization is possible and results in a parametrization of the specialized curve. 
These open sets can be recursively used for decomposing the whole parameter space as it is explained in Section~\ref{sec:decomposition}. 
Applications as described above are presented by using illustrative examples in Section~\ref{sec-application}.

The decomposition of the parameter space, derived from our analysis, depends on the particular method used to compute the genus as well as the rational parametrization.  There exist different methods to deal computationally with the genus: the adjoint curve based method (see e.g.~\cite{walker}, \cite{mauro} and~\cite{myBook}), the method based on the anticanonical divisor (see~\cite{hoeij}) or the method based on Puiseux expansions (see~\cite{PW}), among others. Similarly for the rational parametrization. In this paper we will follow the adjoint curve based method. For the sake of completeness of the paper, we have included an appendix oriented to give the necessary computational details on how the genus and parametrizations are computed.

This manuscript is a self-contained work on the computation of the genus and rational parametrizations of algebraic curves involving parameters. 
Results from various mathematical disciplines are combined for this purpose and presented in a coherent way. 
A rigorous construction of such computable Zariski-open sets were, up to our knowledge, missing in the literature. 
The theorems mentioned in the previous paragraph are novel and can be directly applied in several interesting problems involving parametric curves.
 
\vspace*{2mm}

\noindent \textsf{Notation.} Throughout this paper, the following notation will be used. 
\begin{itemize}
\item We denote by $\aa$ a tuple of parameters.
\item $\K$ is a computable field of characteristic zero.  We represent by $\LL$ the field extension $\LL:=\K(\aa)$. In addition, we consider an algebraic element $\gamma$ over $\LL$. Let $\FF$ be the field $\FF:=\LL(\gamma)$. Furthermore, $K$ represents any field extension of $\K$. We denote by $\overline{K}$ the algebraic closure of $K$, similarly for any field appearing in the paper. 
\item $\S$ is the affine space
\begin{equation}\label{eq-S}
	\S=\overline{\K}^{\,\#(\aa)}
\end{equation}
where $\aa$ will take values.
\item 
For $G\in K[x,y]\setminus K$, we denote by $\Cu(G)$ the plane affine algebraic curve
\begin{equation}\label{eq-CurveAfin} \Cu(G)=\{ p\in \overline{K}^2\,|\, G(p)=0\}. 
\end{equation}
We denote by $G^h(x,y,z)$ the homogenization of $G$, and by $G_x,G_y$ (similarly for $G^{h}_{x}, G^{h}_{y}, G^{h}_{z}$) the partial derivative of $G$ w.r.t. $x$ and $y$ respectively. For a homogeneous polynomial $M\in K[x,y,z]\setminus \{0\}$, $\Cu(M)$ denotes the projective plane curve
\begin{equation}\label{eq-CurveProj} \Cu(M)=\{ p\in \mathbb{P}^2(\overline{K})\,|\, M(p)=0\}. 
\end{equation}
\item 
For polynomials $f,g$ in the variable $t$, and coefficients in an integral domain, we denote by $\mathrm{res}_t(f,g)$ the resultant of $f$ and $g$ w.r.t. $t$.
\item 
Let $\{f_1,\ldots,f_k\}\subset K[\overline{v}]$, where $\overline{v}$ is a tuple of variables. We denote by $\V(f_1,\ldots,f_k)$ the common zero set, over $\overline{K}$, of the polynomials $\{f_1,\ldots,f_k\}$; similarly for $\V(\mathrm{I})$ where $\mathrm{I}$ is an ideal in $K[\overline{v}]$.
\item Let $R\in K[\overline{v}]$, where $\overline{v}$ is a tuple of variables, be a polynomial.  We denote by $\deg(R)$ the total degree of $R$ w.r.t. $\overline{v}$. For a particular variable $v_i$ in $\overline{v}$, we denote by $\deg_{v_i}(G)$ the degree of $G$ w.r.t. the $v_i$. 
\item Let $R\in K(\overline{v})$, where $\overline{v}$ is a tuple of variables, be a rational function. We denote by $\num(R)$ and $\denom(R)$ the numerator and denominator of $R$, respectively, assuming that $R$ is expressed in reduced form; that is, the numerator and denominator are assumed to be coprime. We denote by $\deg(R):=\max\{\deg(\num(R)),\deg(\denom(R))\}$, and  
$\deg_{v_i}(G):=\max\{\deg_{v_i}(\num(R)),\deg_{v_i}(\denom(R))\}$. 
\end{itemize}

\section{Preliminaries on Rational Curves}\label{sec-prel}
In this section we recall some notions and results related to rational (plane) curves that we will be important throughout the paper; for further details we refer to~\cite{myBook}. In addition, some more technical issues, needed in the development of the paper, are included in the appendix. For this purpose, throughout this section, let $G\in K[x,y]\setminus K$.

\subsection{Rational Curves}\label{sec-pre} 
The zero set $\Cu(G)$ of $G$, over the algebraic closure $\overline{K}$, is called the \textit{affine plane curve} associated to $G$ (see~\eqref{eq-CurveAfin}). Taking the homogenization $G^h$ of $G$, we associate to $\Cu(G)$ the \textit{projective plane curve} $\Cu(G^h)$, see~\eqref{eq-CurveProj}. Then $G$ (resp. $G^h)$ is called the \textit{defining polynomial} of $\Cu(G)$ (resp. of $\Cu(G^h)$).

The \textit{degree} of $\Cu(G)$ (similarly for $\Cu(G^h)$) is defined as the total degree of the polynomial $G$ and we denote it as $\deg(\Cu(G))$. Moreover, we say that $\Cu(G)$ is \textit{irreducible} (similarly for $\Cu(G^h)$) if the polynomial $G$ is irreducible over $\overline{K}$. Observe that $\deg(\Cu(G))=\deg(\Cu(G^h))$   and irreducibility of $\Cu(G)$ and $\Cu(G^h)$ are equivalent. In the sequel, we assume that $\Cu(G)$ is irreducible. 
Let us recall that $p\in \Cu(G)$ is a \textit{singular point of $\Cu(G)$ of multiplicity $r\in \N$}, also called \textit{$r$-fold point}, if all partial derivatives of $G$ of order at most $r-1$ vanish at $p$ and, at least, one partial derivative of order $r$ does not vanish. 
The \textit{tangents} to $\Cu(G)$ at an $r$-fold point $p\in \Cu(G)$ are introduced as the lines defined by the linear factors of the sum of all terms in the Taylor expansion of $G$ around $p$ generated by the partial derivatives of $G$, at $p$, of order $r$. 
If all tangents at $p$ are different, $p$ is called \textit{ordinary} otherwise is called \textit{non-ordinary}. We call \textit{character} of a singularity, the fact of being ordinary or non-ordinary. 
The notions of singularity, tangent, and character, are similarly introduced for $\Cu(G^h)$. In the appendix we see how to compute and manipulate the singular locus of the curve. 

For the purposes of the paper, the notion of rational parametrization is important. A \textit{rational (affine) parametrization} of $\Cu(G)$ is a pair $\Pa(t)\in \overline{K}(t)^2 \setminus \overline{K}^2$ such that $G(\Pa(t))=0$. A \textit{rational (projective) parametrization} of $\Cu(G^h)$ is of the form $\Qa(h,t)=(p_1(h,t):p_2(h,t):p_3(h,t))$ where $p_i$ are co-prime homogeneous forms, of the same degree, over $\overline{K}$, not all zero, such that $G^h(\Qa)=0$. Not all curves can be rationally parametrized.  
Curves admitting a rational parametrization are called \textit{rational curves}. 
The rationality of $\Cu(G)$ and $\Cu(G^h)$ are equivalent. Moreover, parametrizations of $\Cu(G)$ and $\Cu(G^h)$ relate each other by means of homogenizing and dehomogenizing. So, in the following we focus on affine parametrizations. 

A parametrization $\Pa(t)$ is called \textit{birational} or \textit{proper} if the map $\overline{K} \dashrightarrow \Cu(G); t\mapsto \Pa(t)$ is injective in a non--empty open Zariski subset of $\overline{K}$ (see e.g.~\cite[Sections 4.2. and  4.4]{myBook}). 

Rational curves can be characterized by means of the genus of the curve. Intuitively speaking the genus measures the difference between the maximum number of singularities (properly counted) the curve, of a fixed degree, may have, and the actual number of singularities it does have. More formally, for an irreducible projective curve $\Cu(G^h)$, the \textit{(geometric) genus} can be defined in various ways, for instance, via the linear space of divisors; see~\cite[Definition 3.4]{myBook},~\cite[Section 8.6]{Fulton} or~\cite[Section 8.2]{mauro} for more details. The genus of $\Cu(G)$ is defined as the genus of $\Cu(G^h)$. Now rational curves are precisely the irreducible curves of genus zero (see e.g. Theorem 4.63 in~\cite{myBook} or~\cite[page 210]{mauro}). As commented in the introduction, the genus can be computed in different ways. 
There exist algorithmic methods to compute the genus of an algebraic curve and to determine, when the genus is zero, a rational parametrization of the curve (see e.g.~\cite{hoeij}, \cite{schicho}, \cite{SW91}, \cite{SW97}, \cite{myBook}). Within this paper, we make use of the adjoint curves based method for parametrizing curves which is summarized in Appendix~\ref{App}.

\subsection{Fields of Parametrization}\label{sec-FieldOfparametrization}

In general, if one computes a parametrization $\Pa(t)$ of $\Cu(G)$, the ground field, that is the smallest field where the coefficients of $G$ belong to, has to be extended (see e.g. Sections 4.7. and 4.8. in~\cite{myBook}). A field is called a \textit{parametrizing field} or \textit{field of parametrization} of $\Cu(G)$ if there exists a parametrization of $\Cu(G)$ with coefficients in it. In this subsection, we recall some results on fields of parametrization. 

Hilbert-Hurwitz Theorem (see~\cite{HH}) plays a fundamental role in this context. Thus, let us briefly comment it. For this purpose, we will use the notion of adjoint curve (see e.g. Definition 4.55 and 4.56 in~\cite{myBook} or Subsection~\ref{subsec:param} in the appendix).  
Given a projective rational plane curve $\Cu(G^h)$ of degree $d$, Hilbert-Hurwitz Theorem states that almost all 3-tuples of $(d-2)$-degree adjoints, to $\Cu(G^h)$, define a rational map from $\mathbb{P}^2$ to $\mathbb{P}^2$, that maps birationally $\Cu(G^h)$ onto a $(d-2)$--degree projective irreducible plane curve, say $\Cu(G^*)$. Now, since the ground field of the original curve and that of the adjoints coincides, say $K$, (see e.g. Theorem 4.66 in~\cite{myBook}), the ground field of $\Cu(G^*)$ is also $K$. Furthermore, since birational maps preserve the genus of the curve (see e.g.~\cite[Theorem 7.2.2 or Prop. 8.5.1]{mauro}), one has that $\Cu(G^*)$ is also rational. Therefore, one may apply again the theorem to $\Cu(G^*)$. 
After finitely many applications of the theorem one gets a birational map sending $\Cu(G^h)$ onto a conic (if $d$ is even) or a line (if $d$ is odd) with the same ground field. Finally, taking into account that a line can be parametrized over the ground field, and a conic over a field extension of the ground field of degree at most two, one concludes that there always exists a field of parametrization of $\Cu(G^h)$ that involves a field extension of $K$ of degree at most two (see Corollary 5.9. in~\cite{myBook} for further details). 
Indeed, this field extension, of degree at most two, is the field extension used in Step (4) of the parametrization computation (see Subsection~\ref{subsec:param} in the appendix).

Let us now focus on the case where the ground field is $\LL:=\K(\aa)$; see notation above. Let $G\in \LL[x,y]$ be an irreducible (over $\overline{\LL}$) non-constant polynomial, and let $\mathcal{C}(G)$ be rational. We analyze the fields of parametrization of $\Cu(G)$. We observe that if the degree two field extension is $\LL(\alpha)$, with minimal polynomial $t^2+b\,t+c\in \LL[t]$, then $\LL(\alpha)=\LL(\beta)$ where $\beta=\alpha+b/2$ has the minimal polynomial $t^2+c-b^2/4$.
Therefore, taking into account the above discussion on Hilbert-Hurwitz Theorem, the following holds (see also~\cite[Corollary 5.9]{myBook}).

\begin{theorem}\label{theorem:HHext} \
\begin{enumerate}
\item If $\deg(\Cu(G))$, is odd then $\LL$ is a field of parametrization.
\item If $\deg(\Cu(G))$ is even then either $\LL$ is a field of parametrization or there exists $\delta\in \overline{\LL}$ algebraic over $\LL$, with minimal polynomial $t^2-\alpha\in \LL[t]$, such that $\LL(\delta)$ is a field of parametrization of $\Cu(G)$.
\end{enumerate}
\end{theorem}

\begin{remark}
Observe that the previous result is valid taking $\LL$ as any field extension of $\K$.
\end{remark}

The case where $\mathbf{a}$ contains a single element admits a particular treatment because of Tsen's Theorem; we refer to~\cite{tsen} for this topic.

\begin{corollary}\label{cor:TsenN2}
If $\#(\aa)=1$, then $\LL$ is a field of parametrization of $\Cu(G)$.
\end{corollary}
\begin{proof}
By Hilbert-Hurwitz Theorem (see e.g.~\cite[Theorem 5.8]{myBook} and as explained above), $\Cu(G)$ is $\LL$--birationally equivalent to either a line or a conic. 
In the line case, the result is clear. In the conic case, since $\#(\aa)=1$, the result follows from Tsen's Theorem (see e.g.~\cite[Corollary 1.11]{shafa}).
\end{proof}

\begin{remark}
The proof of Tsen's Theorem provides a method for computing a regular point on the conic with coordinates in $\LL$. An alternative approach for computing this point can be found in~\cite{HW} and~\cite{vo}.
\end{remark}

\begin{remark}\label{rem:TsenGeneralization}
In the following section we work with $G\in \K[\aa,\gamma][x,y]$ where $\gamma$ is algebraic over $\K(\aa)$. 
In the case where $\#(\aa)=1$, we can view $\gamma$ as the only parameter and write $\aa$ in terms of $\gamma$. 
More precisely, let $M(\aa,c) \in \K(\aa)[c]$ be the minimal polynomial of $\gamma$. We can view $M$ as a rational expression in $\aa$ and consider its numerator    $$H(\mathbf{b}):=\num(M)(\aa=\mathbf{b},c=\gamma) \in \K(\gamma)[\mathbf{b}]$$ as polynomial in $\mathbf{b}$ with the root $\aa$. 
Thus, $\K(\gamma,\aa):\K(\gamma)$ is a field extension of degree $d \le \deg(H)$. 
If $d=1$, by Corollary~\ref{cor:TsenN2}, $\K(\gamma)$ is a field of parametrization of $\Cu(G)$.
\end{remark}

\begin{remark}
In Corollary~\ref{cor:TsenN2}, we have seen that if $\#(\aa)=1$ then $\LL$ is a field of parametrization. The following example shows that if $\#(a)>1$, in general, $\LL$ is not a field of parametrization. We consider the conic defined by $ {G}:=a_1 x^2+a_2 y^2-1$. We prove that $\Cu(G)$ does not have a parametrization over $\C(a_1,a_2)$. Let us assume that $\C(a_1,a_2)$ is a field of parametrization of $\Cu(G)$, then $\Cu(G)$ has infinitely many points in $\C(a_1,a_2)^2$. $\Cu(G)$ parametrizes properly as 
\[ \Pa:=\left( \sqrt{a_1}\, \frac{1-t^{2}}{t^{2}+1}, \sqrt{a_2} \,\frac{2 t}{t^{2}+1}
\right), \]
which inverse is
\[ \Pa^{-1}(x,y)=\frac{\sqrt{a_{2}}\, \left(\sqrt{a_{1}}+x \right)}{\sqrt{a_{1}}\, y}.\]
So, there are infinitely many points in $\Cu(G)\cap \C(a_1,a_2)^2$ that are injectively reachable, via $\Pa$, for $t\in \C(\sqrt{a_1},\sqrt{a_2})$. Indeed, note that all points of $\Cu(F)$, with the exception of $(-\sqrt{a_1},0)$, are reachable by $\cP$. Let $t_0\in \C(\sqrt{a_1},\sqrt{a_2}) \setminus \{ 0,\pm \mathrm{i}\}$ be one of these parameter values; say $\cP(t_0)=(x_0,y_0)\in \C(a_1,a_2)^2$. 
Then $t_{0}^2 = (\sqrt{a_1}+x_0)/(\sqrt{a_1}-x_0) \in \C(\sqrt{a_1},a_2).$ For $x_0 \ne 0$, it holds that $\sqrt{a_1}+x_0$, $\sqrt{a_1}-x_0$ are coprime (seen as polynomials in $\sqrt{a_1}$) and $$t_0 = \pm \sqrt{\tfrac{\sqrt{a_1}+x_0}{\sqrt{a_1}-x_0}} \notin \C(\sqrt{a_1},\sqrt{a_2}),$$ a contradiction. 
For $x_0=0$ the curve-points $(x_0,y_0)=(0,\pm 1/\sqrt{a_2})$ are not in $\C(a_1,a_2)$.
\end{remark}

\section{Specializations}\label{sec-specialization}
Throughout the paper, we will specialize the tuple of parameters $\aa$ taking values in $\S$ (see~\eqref{eq-S}). We will write $\aao$ to emphasize that the parameters in $\aa$ have been substituted by elements in $\overline{\K}$. In the following we discuss different aspects on the specializations.

\subsection{General statements}
The elements in $\overline{\K}(\aa)$ are assumed to be represented in reduced form; that is, the numerator and denominator are assumed to be coprime. Then, for $f:=p(\aa)/q(\aa)\in \overline{\K}(\aa)$, where by assumption $\gcd(p,q)=1$, and for $\aao\in \S$ (see~\eqref{eq-S}) such that $q(\aao)\neq 0$, we denote by $f(\aao)$ the $\overline{\K}$--element $p(\aao)/q(\aao)$.

We may need to work in the finite field extension $\FF=\LL(\gamma)=\K(\aa)(\gamma)$. Let $p(\aa,t) \in \K(\aa)[t]$, of degree $k$ in $t$, be the minimal polynomial of $\gamma$. We might simply write $p(t)$ instead of $p(\aa,t)$ and express it as
\begin{equation}\label{eq-alpha}
p(t)= t^k+\frac{N_{k-1}(\aa)}{D_{k-1}(\aa)}\,t^{k-1} + \cdots + \frac{N_{0}(\aa)}{D_{0}(\aa)}, \,\,\text{with $N_i,D_i \in \K[\aa]$ and $\gcd(N_i,D_i)=1$.} 
\end{equation}
Then, for $\aao\in \S$ such that all $D_i(\aao)\neq 0$, we denote by $\gamma^0$ the algebraic element, over $\K(\aao)$, defined by an irreducible factor of
$$p(\aao,t)= t^k+\frac{N_{k-1}(\aao)}{D_{k-1}(\aao)} \,t^{k-1}+ \cdots + \frac{N_{0}(\aao)}{D_{0}(\aao)} \in \K(\aao)[t]\subset \overline{\K}[t].$$
 The below reasonings are valid, independently of the irreducible factor taken to define $\gamma^0$. Then, 
for an element $f \in \FF$, specialized at $\aao\in \S$, we might simply write $f(\aao)$ instead of $f(\aao,\gamma^0)$.

\begin{definition}\label{def-gamma}
We define the open subset $\Omega_{\gamma}:=\S \setminus \V(D)$ where $D:=\lcm(D_0,\ldots,D_{k-1})$.
\end{definition}

Clearly for $\aao\in \Omega_\gamma$, $\gamma^{0}$ is well--defined. The elements in $\FF$ are assumed to be expressed in canonical form; that is, $f\in \FF$ is expressed as
\begin{equation}\label{eq-f}
f=\sum_{i=0}^{k-1}\dfrac{U_i(\aa)}{W(\aa)}\gamma^i
\end{equation}
where $U_i,W\in \K[\aa]$ and $\gcd(U_1,\ldots,U_{k-1},W)=1$. In addition, the coefficients of polynomials in $\FF[\overline{v}]$, w.r.t. to the tuple of variables $\overline{v}$, are also supposed to be written in canonical form. 
By abuse of notation, we will also denote by $f(\aa,t)$ the polynomial in $\LL[t]$ obtained by replacing in~\eqref{eq-f} the element $\gamma$ by the variable $t$.
Moreover, for $f$ as in~\eqref{eq-f}, we denote by $\No(f)$ the element
$$\No(f):= \prod f(\aa,\gamma_i)$$
where the product is taken over all roots $\gamma_i$ in $\overline{\LL}$ of the minimal polynomial of $\gamma$,  say $p(\aa,t)$ (see~\eqref{eq-alpha}). Since $p$ is monic, taking into account the expression of the resultant as the product of the evaluations of one of the polynomials in the roots of the other, we get that, up to sign, $\No(f)=\Res_t(f(\aa,t),p(\aa,t))$. In particular, $\No(f)\in \LL$.

\begin{lemma}\label{lem:N}
Let $f$ be as in~\eqref{eq-f}. Let $\aao\in \Omega_\gamma$ be such that $W(\aao)\neq 0$ (see Def.~\ref{def-gamma}). If $f(\aao)=0$, then $\No(f)(\aao)=0$.
\end{lemma}
\begin{proof}
$D(\aao)W(\aao)\neq 0$. So $\gamma^0$, $f(\aao)$ and $\No(f)(\aao)=\Res_t(f(\aao,t),p(\aao,t))$ are well-defined. Since  $f(\aao,\gamma^0)=0=p(\aao,\gamma^0)$, we obtain $\No(f)(\aao)=0$.
\end{proof}

\begin{definition}\label{def-open1}
Let $H\in \FF [\overline{v}]$, where $\overline{v}$ is a tuple of variables. Let $S$ be the set of all non-zero coefficients of $H$ w.r.t. $\overline{v}$. Let
$$\mathcal{D}(H):=\lcm(\{\mathrm{denom}(C)\,|\, C\in S\})\in {\K}[\aa],$$
 note that $C\in S\subset \FF$   is  a rational function in $\aa,\gamma$  expressed in canonical form and, hence, its denominator is in $\K[\aa]$. And let
$$\mathcal{V}(H ):=\{  \mathrm{numer}(\No(\mathrm{numer}(C)))\,|\, C\in S \} \subset {\K}[\aa],$$
 note that $\No(\mathrm{numer}(C))\in \LL$ and, hence its numerator is in $\K[\aa]$. 
 We associate to $H$ the following open subsets
\begin{enumerate}
	\item $\Omega_{\mathrm{def}(H)}:=\Omega_\gamma \cap\left( \S\setminus \V(\mathcal{D}(H))\right).$
	\item $\Omega_{\mathrm{nonZ}(H)}:=\Omega_{\mathrm{def}(H)} \cap \left(\S\setminus \V(\mathcal{V}(H ))\right).$
\end{enumerate}
\end{definition}

\begin{remark}
Throughout the paper, we will define several open subsets of $\S$. All these open subsets will be included in $\Omega_{\gamma}$ (for the corresponding algebraic element $\gamma$). So, we observe that $\gamma^0$ will always be well--defined.
\end{remark} 

The next lemma justifies the previous definitions.

\begin{lemma}\label{lem:defVan}
Let $H\in \FF[\overline{v}]$, where $\overline{v}$ is a tuple of variables. It holds that
\begin{enumerate}
	\item If $\aao\in \Omega_{\mathrm{def}(H)}$ then $H(\aao,\gamma^0,\overline{v})$ is well-defined.
	\item If $\aao\in \Omega_{\mathrm{nonZ}(H)}$ then $H(\aao,\gamma^0,\overline{v})\neq 0$.
\end{enumerate}
\end{lemma}
\begin{proof}
(1) Let $\aao\in \Omega_{\mathrm{def}(H)}\subset \Omega_\gamma$. Then, $\gamma^0=\gamma(\aao)$ is well--defined, and the result follows from the definition of $\mathcal{D}(H)$.\\
(2) Let $\aao\in \Omega_{\mathrm{nonZ}(H)}\subset \Omega_{\mathrm{def}(H)}$. Then, by (1), $H(\aao,\gamma^0,\overline{v})$ is well--defined. Furthermore, there exists a coefficient of $H$ w.r.t. $\overline{v}$, say $C(\aa,\gamma)$, such that $\No(\mathrm{numer}(C))(\aao)\neq 0$. Since $D(\aao)\neq 0$ (see Def.~\ref{def-gamma}) and the denominator of $C$ does not vanish at $\aao$, by Lemma~\ref{lem:N}, we get that $C(\aao,\gamma^0)\neq 0$. So, $H(\aao,\gamma^0,\overline{v})\neq 0$.
\end{proof}

The following lemma is an adaptation of Lemma 3 in~\cite{SW01} to our case, and will be used to control the birationality of a curve  parametrization $\Pa(\aa,t)$ under specializations of $\aa$.

\begin{definition}\label{def-gcd}
Let $f_1,f_2\in \FF[u][v]\setminus\{0\}$ for $i\in \{1,2\}$, where $u,v$ are variables. Let $f_i=f_{i}^{*} \, g$, for $i\in \{1,2\}$, where $g=\gcd_{ \FF[u][v]}(f_1,f_2)$. Let $A_i\in \FF[u]$ be the leading coefficient of $f_i$ w.r.t. $v$ for $i\in \{1,2\}$ and $B\in \FF[u]$ the leading coefficient of $g$ w.r.t. $v$. Let $R=\res_{v}(f_{1}^{*},f_{2}^{*})\in \FF[u]$.
Let
\[ \begin{array}{l}
\Omega_1:= \Omega_{\mathrm{def}(f_1)} \cap \Omega_{\mathrm{def}(f_2)} \cap  \Omega_{\mathrm{def}(f_{1}^{*})} \cap \Omega_{\mathrm{def}(f_{2}^{*})}\cap \Omega_{\mathrm{def}(g)} \cap \Omega_{\mathrm{def}(R)},\\
\Omega_2:=\Omega_{\mathrm{nonZ}(A_1)} \cap \Omega_{\mathrm{nonZ}(A_2)} \cap \Omega_{\mathrm{nonZ}(B)}\cap \Omega_{\mathrm{nonZ}(R)}.
\end{array}
\]
We define the set
\[ \Omega_{\mathrm{gcd}_{\FF[u][v]}(f_1,f_2)}:= \Omega_1\cap \Omega_2. \]
\end{definition}

\begin{remark}
In the sequel, unless there is a risk of ambiguity, we will omit in the notation of the gcds, and of the open subset $\Omega_{\gcd}$, the polynomial ring where the gcd is taken.
\end{remark}

  In the remaining part of this section, we might write for the specialization of a given polynomial, or rational function, $P$ at $\aao$ simply $P^0$.  

\begin{lemma}\label{lem-gcd}
Let $f_1,f_2,g$ be as in Def.~\ref{def-gcd}.
For
$\aao\in  \Omega_{\mathrm{gcd}_{\FF[u][v]}(f_1,f_2)},$ it holds that
\[ \lambda(u) \, g(\aao,\gamma^0,u,v)=  \gcd_{\overline{\K}[u][v]}(f_1(\aao,\gamma^0,u,v),f_2(\aao,\gamma^0,u,v)), \]
with $\lambda(u)\in \overline{\K}[u]\setminus \{0\}$. Moreover,  $\deg_v(g(\aao,\gamma^0,u,v))=\deg_v(g(\aa,\gamma,u,v)).$
\end{lemma}
\begin{proof}
Let $f_{i}^{*},A_i,B,R$ be as in Def.~\ref{def-gcd} and let $\aao\in \Omega_{\mathrm{gcd}(f_1,f_2)}$.  $f_{i}^{0},g^0,f_{i}^{*\,0},A_{i}^{0},B^{0},R^0$  are all well--defined because $\aao\in \Omega_1$ (see Def.~\ref{def-gcd} and Lemma~\ref{lem:defVan}~(1)). So,
\begin{equation}\label{eq-gcd1}
f_{i}^{0}(u,v)=f_{i}^{*\,0}(u,v)g^0(u,v).
\end{equation}
Moreover, since $\aao\in \Omega_2$ (see Def.~\ref{def-gcd}), by Lemma~\ref{lem:defVan}~(2), $B^0\neq 0$, $R^0\neq 0$. Furthermore, $f_{i}^0,g^0$  preserve the degree in $v$ and, in particular, are also non--zero.
This implies that $f_{i}^{*\,0}$ are non--zero too. From~\eqref{eq-gcd1}, one has that, up to multiplication by non-zero elements in $\overline{\K}$, 
\[  \gcd_{\overline{\K}[u][v]}(f_{1}^0,f_{2}^{0})=  \, \gcd_{\overline{\K}[u][v]}(f_{1}^{*\,0},f_{2}^{*\,0}) \, g^0. \]
Let us prove that $\delta:=\deg_v(\gcd_{\overline{\K}[u][v]}(f_{1}^{*\,0},f_{2}^{*\,0}))=0$, in which case taking $\lambda(u)$ as this gcd, the gcd equality in the statement of the lemma would hold. Indeed, let $\delta>0$. Then, if $\tilde{R}(u)=\Res_v(f_{1}^{*\,0},f_{2}^{*\,0})$, we get that $\tilde{R}$ is zero (see e.g.~\cite[Corollary page 288]{Geddes}). However, since $\aao\in \Omega_2$ (see Def.~\ref{def-gcd}), by Lemma~\ref{lem:defVan} (2), $A_{1}^{0},A_{2}^{0}$ are not zero and, hence, the leading coefficients of $f_{i}^{*}$ do not vanish either at $\aao$. Thus, by~\cite[Lemma 4.3.1]{winkler}, $R^0=\tilde{R}(u)$. But $R^0\neq 0$ which is a contradiction. So, $\delta=0$. In addition, since $B^0\neq 0$, $\deg_v(g^0)=\deg_v(g).$
\end{proof}
 
If the polynomials are univariate over a field, Lemma \ref{lem-gcd} can be simplified as follows.

\begin{corollary}\label{cor-lemma-gcd}
Let $f_1,f_2\in \FF[v]\setminus\{0\}$ for $i\in \{1,2\}$. Let  $g=\gcd_{\FF[v]}(f_1,f_2)$. For $\aao\in \Omega_{\mathrm{gcd}_{\FF[v]}(f_1,f_2)},$ it holds that, up to multiplication by non-zero elements in $\overline{\K}$,
\[  g(\aao,\gamma^0,v)=\gcd_{\overline{\K}[v]}(f_1(\aao,\gamma^0,v),f_2(\aao,\gamma^0,v)). \]  
Moreover, $\deg_v(g(\aao,\gamma^0,v))=\deg_v(g(\aa,\gamma,v)).$
\end{corollary}

We generalize these results to several univariate polynomials with coefficients in $\FF$.

\begin{definition}\label{def:gcd-several}
Let $f_1,\ldots,f_r\in \FF[v]\setminus\{0\}$, $g=\gcd_{\FF[v]}(f_1,\ldots,f_r)$ and $A$ the leading coefficient of $g$ w.r.t. $v$. We consider the polynomial $f_Z:=f_2+f_3Z+\cdots +f_{r} Z^{r-2}\in \FF[Z][v]$ where $Z$ is a new variable. We define
\[ \Omega_{\mathrm{gcd}_{\FF[v]}(f_1,\ldots,f_r)}:= \Omega_{\mathrm{gcd}_{\FF[Z][v]}(f_1,f_Z)} \cap \Omega_{\mathrm{def}(g)}\cap \Omega_{\mathrm{nonZ}(A)}.\]
\end{definition}
 
\begin{remark}\label{rem-gcd-generalized}
Observe that if $r=2$ in Def.~ \ref{def:gcd-several}, then Def.~\ref{def-gcd} and~\ref{def:gcd-several} coincide. In addition, since $f_1$ does not depend on $Z$,  there exists $\mu\in \FF\setminus \{0\}$ such that $\gcd_{\FF[Z][v]}(f_1,F_Z)=\mu \gcd_{\FF[v]}(f_1,\ldots,f_r)$. Similarly if the polynomials are taken in $\overline{\K}[v]$.
\end{remark}

\begin{theorem}\label{theorem-lemma-gcd-several-pol}
Let $f_1,\ldots,f_r$ be as in Def.~\ref{def:gcd-several}. Let $g=\gcd_{\FF[v]}(f_1,\ldots,f_r)$. For $\aao\in \Omega_{\mathrm{gcd}_{\FF[v]}(f_1,\ldots,f_r)},$ it holds that, up to multiplication by non-zero elements in $\overline{\K}$,
\[ g(\aao,\gamma^0,v)=\gcd_{\overline{\K}[v]}(f_1(\aao,\gamma^0,v),\ldots,f_r(\aao,\gamma^0,v)). \]
Moreover, $\deg_v(g(\aao,\gamma^0,v))=\deg_v(g(\aa,\gamma,v)).$
\end{theorem}
\begin{proof}
 Let $g^*=\gcd_{\FF[Z][v]}(f_1,f_Z)$, and $\aao\in \Omega_{\mathrm{gcd}_{\FF[v]}(f_1,\ldots,f_r)}$. 
By Remark~\ref{rem-gcd-generalized},  $g^*\in \FF[v]$. 
By Lemma~\ref{lem-gcd}, there exists $\lambda\in \overline{\K}[Z]\setminus \{0\}$ such that 
\[ \lambda(Z)\,g^{*\,0} = \gcd_{\overline{\K}[Z][v]}(f_{1}^{0},f_{Z}^{0})\in \overline{\K}[v], \]
and $\deg_v(g^{*\,0})=\deg_v(g^*).$
Note that, since $g^{*\,0}$ does not depend on $Z$, then $\lambda \in \overline{\K}\setminus \{0\}$. On the other hand, by Remark \ref{rem-gcd-generalized}, there exists $\mu\in \FF\setminus \{0\}$ such that $g^*=\mu g$. Furthermore, $\mu^0:=\mu(\aa^0,\gamma^0)$ is well defined. Moreover,  since the leading coefficients of $g^*$ and $g$ do no vanish when specialized, then $\mu^0\neq 0$. Thus, up to multiplication by non-zero elements in $\overline{\K}$, it holds that
\[    g^0=g^{*\,0}=\gcd_{\overline{\K}[Z][v]}(f_{1}^{0},f_{Z}^{0})=\gcd_{\overline{\K}[v]}(f_{1}^{0},\ldots,f_{r}^{0}). \]
Moreover, by definition of the open set, the leading coefficient of $g$ w.r.t. $v$ does not vanish at $\aa^0$. Therefore, $\deg_v(g^0)=\deg_v(g ).$
\end{proof}

Our next step is to analyze the squarefreeness.

\begin{definition}\label{def:sqfree}
Let $f\in \FF[v]\setminus \FF$ be squarefree. Let $R$ be the discriminant of $f$ w.r.t. $v$ and let $A$ be the leading coefficient of $f$ w.r.t. $v$. We define the open subset
\[ \Omega_{\mathrm{sqfree}(f)}:=\Omega_{\mathrm{def}(f)} \cap \Omega_{\mathrm{nonZ}(R)}\cap \Omega_{\mathrm{NonZ}(A))}.\]
\end{definition}

\begin{lemma}\label{lem:sqfree}
Let $f\in \FF[v]\setminus \FF$ be squarefree. If $\aao\in  \Omega_{\mathrm{sqfree}(f)}$, then $\deg_{v}(f(\aa,\gamma,v))= \deg_{v}(f(\aao,\gamma^0,v))$ and $f(\aao,\gamma^0,v)$ is squareefree.
\end{lemma}
\begin{proof}
Since $\aao\in\Omega_{\mathrm{def}(f)}$, by Lemma~\ref{lem:defVan}, $f(\aao,\gamma^0,v)$ is well--defined. Since $\aao\in\Omega_{\mathrm{NonZ}(A)}$, $f(\aao,\gamma^0,v)\neq 0$ and the equality of the degrees holds.
Since $\aao\in \Omega_{\mathrm{nonZ}(R)}$, also by Lemma~\ref{lem:defVan},  $R(\aao,\gamma^0,v)$ is well-defined and non--zero. Since $A(\aao,\gamma^0,v)\neq 0$, by~\cite[Lemma 4.3.1]{winkler}, the discriminant of $f(\aao,\gamma^0,v)$ is not zero. Then, by~\cite[Theorem 4.4.1]{winkler}, $f(\aao,\gamma^0,v)$ is squarefree.
\end{proof}

\subsection{Specialization of the curve defining polynomial}\label{subsec:defpol}
In this subsection, we deal with the specialization of defining polynomials of irreducible plane curves. Let $G\in \FF[x,y]\setminus \FF$ be irreducible over $\overline{\FF}$ of total degree $d$. 
In the following, let $G$ be written as
\begin{equation}\label{eq-G}
G=g_d(x,y)+\cdots+g_0(x,y)
\end{equation}
where $g_i$ is either the zero polynomial or a form of degree $i$.

\begin{definition}\label{def:omegaG}
We associate to $G$ the open subset (see~\eqref{eq-G})
$\Omega_G:=\Omega_{\mathrm{def}(G)}\cap \Omega_{\mathrm{nonZ}(g_d)}.$
\end{definition}

\begin{lemma}\label{lem:omegaG} Let $G$ be as above and let $\aao\in \Omega_G$. Then
\begin{enumerate}
\item $G(\aao,\gamma^0, x,y)$ is well--defined and $\deg(G(\aao,\gamma^0, x,y))=\deg(G)$;
\item $G(\aao,\gamma^0,  x,y)^{h}=G^{h}(\aao, \gamma^0, x,y,z)$;
\item the partial derivatives of $G^h$, of any order, specialize properly.
\end{enumerate}
\end{lemma}
\begin{proof}
Since $\aao\in \Omega_{\mathrm{def}(G)}$, by Lemma~\ref{lem:defVan}, $G(\aao,\gamma^0, x,y)$ is well--defined and, since $\aao\in \Omega_{\mathrm{nonZ}(g_d)}$, the equality on the degree holds. Now, since the total degree $d=\deg(G)$ is preserved under the specialization, $G(\aao,\gamma^0, x,y)^h=(g_d(\aao,\gamma^0,x,y)+\cdots+g_0(\aao,\gamma^0,x,y))^h=g_d(\aao,\gamma^0,x,y)+\cdots+g_0(\aao,\gamma^0,x,y)z^{d}=G(\aao,\gamma^0,x,y,z)$. Finally, since derivatives of $G^h$ do not introduce denominators, all derivatives are well-defined when specializing at $\aao$, and the derivative of the specialization is the specialization of the derivative.
\end{proof}

\subsection{Specialization of families of points}\label{subsec-spec-families}
Let us now deal with the specialization of conjugate families of points appearing in the standard decomposition of the singular locus of the curve. In Subsections~\ref{subsec-families} and~\ref{subsec-standard-decomp} of the appendix, a description of these concepts appears; see in particular Definitions~\ref{def-family} and~\ref{def-stand-decomp} and Remark~\ref{rem-standard-decomp-final}.

Let $G$ and $G^h$ be as in Subsection~\ref{subsec:defpol}. For $\aao\in \S$ such that $G(\aao,x,y)\not\in \overline{\K}$, $G(\aao,x,y)$ defines a plane curve over $\overline{\K}$. Let $\Cu(G)$ denote the curve defined by $G(\aa,\gamma,x,y)$ over $\overline{\FF}$, and $\Cu(G,\aao)$ the curve defined by $G(\aao,\gamma^0,x,y)$ over $\overline{\K}$. Similarly for the projective curves.

Conjugate families of a curve are referred to a field. The conjugate families of $\Cu(G^h)$ will be over $\FF$. When we specialize $\aa$ we need to have a reference field where the conjugation of the points is defined. This motivates the following definition.

\begin{definition}\label{def-Ka} For $\aao\in \S$, we define $\K_{\aao}$ as the smallest subfield of $\overline{\K}$ containing the coefficients of $G(\aao,\gamma^0,x,y)$. Moreover, if $\mathcal{F}=\{(f_1:f_2:f_3)\}_{m}$ is an $\FF$--conjugate family, and $\aao\in \S$ is such that $\gamma^0, f_1(\aao,t), f_2(\aao,t), f_3(\aao,t), m(\aao,t)$ are well--defined, we denote by $\mathcal{F}(\aao)$ the specialization of $\mathcal{F}$ at $\aao$ (and $\gamma^0$).
\end{definition}

Let $\Stand(G^h)$ denote the $\FF$--standard decomposition of the singular locus of $\Cu(G^h)$ obtained using the process described in Subsection~\ref{subsec-standard-decomp}; see also Def.~\ref{def-stand-decomp} and Remark~\ref{rem-standard-decomp-final}.
\begin{equation}\label{eq-SingularLocus}
\Stand(G^h)=\bigcup_{m\in \mathcal{A}_a} \{(f_{1,m}:f_{2,m}:1)\}_{m} \, \cup \,  \bigcup_{m\in \mathcal{A}_{\infty}} \{(L_{1,m}:L_{2,m}:0)\}_{m}
\end{equation}
where $f_{i,m}(t),m(t)\in \FF[t]$, $L_{i,m}(t)\in \K[t]$ with $\deg(L_{i,m}(t))\leq 1$, $\gcd(L_{1,m}(t),L_{2,m}(t))=1$, and where $\mathcal{A}_a$ and $\mathcal{A}_{\infty}$ are finite sets of irreducible polynomials in $\FF[t]$. By abuse of notation, we will write $\mathcal{F}\in\Stand(G^h)$ to refer to the families in $\Stand(G^h)$.

\begin{definition}\label{def-defFam} Let $\mathcal{F}:=\{(f_{1,m}:f_{2,m}:1)\}_{m}\in \Stand(G^h)$ be an irreducible $\FF$-family of affine singularities of $\Cu(G^h)$ (see~\eqref{eq-SingularLocus}).
Let $A$ be the product of the leading coefficients w.r.t. $t$ of $m, f_{1,m}, f_{2,m}$. We associate to $\mathcal{F}$ the open set
\[ \Omega_{\mathrm{def}(\mathcal{F})}:= \displaystyle{  \Omega_{\mathrm{gcd}_{\FF[t]}(f_{1,m},f_{2,m},m)} \cap \Omega_{\mathrm{sqfree}(m)} \cap \Omega_{\mathrm{NonZ}(A)} \cap \, \Omega_{G}}. \]
 Let $\mathcal{F}:=\{(L_{1,m}:L_{2,m}:0)\}_{m}\in \Stand(G^h)$ be an irreducible $\FF$-family of singularities of $\Cu(G^h)$ at infinity (see~\eqref{eq-SingularLocus}).
Let $A$ be the leading coefficient of $m$ w.r.t. $t$. We associate to $\mathcal{F}$ the open set
\[ \Omega_{\mathrm{def}(\mathcal{F})}:=  \Omega_{\mathrm{gcd}_{\FF[t]}(L_{1,m},L_{2,m},m)} \cap \displaystyle{\Omega_{\mathrm{sqfree}(m)} \cap \Omega_{\mathrm{NonZ}(A)} \cap \, \Omega_{G}}. \]
\end{definition}

We start our analysis with a technical lemma.  In the remaining part of this section, we might write for the specialization of a given polynomial, or rational function, $P$ at $\aao$ simply $P^0$.  

\begin{lemma}\label{lem:spec-mod} Let $H,m\in \FF[t]$. Let $R$ be the remainder of the division of $H$ by $m$ w.r.t. $t$, and let $A$ be the leading coefficient of $m$ w.r.t. $t$. If $H(\aao,t), m(\aao,t)$ are well--defined and $A(\aao)\neq 0$, then $H(\aao,t)=R(\aao,t)$ mod $m(\aao,t)$.
\end{lemma}
\begin{proof}
Let $Q$ be the quotient of the divison of $H$ by $m$ w.r.t. $t$. So, $H=Q \cdot m +R$ with $\deg_t(R)<\deg_t(m)$.
Since $H^0, m^0$ are well--defined,$\gamma^0$ exists. Since $A^0\neq 0$, then $Q^0,R^0$ are well--defined. Moreover, $\deg_t(m)=\deg_t(m^0)$. Then $H^0=Q^0 m^0 +R^0$ with $$\deg_t(R^0)\leq \deg_t(R)<\deg_t(m)=\deg_t(m^0).$$ This concludes the proof.
\end{proof}

\begin{lemma}\label{lem:defF}
Let $\mathcal{F}\in\Stand(G^h)$ be an irreducible $\FF$--family of $\Cu(G^h)$ (see~\eqref{eq-SingularLocus}). If $\aao\in \Omega_{\mathrm{def}(\mathcal{F})}$, then $\mathcal{F}(\aao)$ is a $\K_{\aao}$--family of points of $\Cu(G^h,\aao)$ and $\#(\mathcal{F})=\#(\mathcal{F}(\aao))$.
\end{lemma}
\begin{proof} 
Let $\mathcal{F}=\{(f_{1,m}:f_{2,m}:1)\}_{m}$.
Let us first show that $\mathcal{F}(\aao)$ is a $\K_{\aao}$-family of points of $\Cu(G^h,\aao)$. Since $\aao\in \Omega_{\mathrm{gcd}(f_{1,m},f_{2,m},m)}$ (see Def.~\ref{def:gcd-several}) then $\gamma^0$, $f_{i,m}^0$ and $m^0$ are well-defined and, by Theorem \ref{theorem-lemma-gcd-several-pol}, the $\gcd(f_{1,m}^0, f_{2,m}^0, m^0)=1$.
Furthermore, since $\aao\in \Omega_{\mathrm{NonZ}(A)}$ (see Def.~\ref{def-open1}), the degree of $f_{i,m}$ and $m$ is preserved under the specialization. In addition, since $\aao\in \Omega_{\mathrm{sqfree}(m)}$, it holds that $m^0$ is squarefree (see Lemma~\ref{lem:sqfree}). So the conditions in Def.~\ref{def-family} hold. Furthermore, note that, after specialization, all polynomials are over $\K_{\aao}$. So, $\mathcal{F}(\aao)$ is a family over $\K_{\aao}$. It remains to prove that the points in $\mathcal{F}(\aao)$ are in the specialized curve.
Since $\aao\in \Omega_{G}$, by Lemma~\ref{lem:omegaG}, it holds that 
$(G^0)^h=(G^h)^0$. 
Let $T(\aa,t)=G^h(\aa,f_{1,m},f_{2,m},1)$. Since $\mathcal{F}$ is a family of points in $\Cu(G^h)$, it holds that $T=0$ mod $m$. Since $(G^h)^0$    and $f_{i,m}^0$    are well--defined, then $T^0$ is well--defined too. We know that $m^0$ is well--defined and that the leading coefficient of $m$ in $t$ does not vanish after specialization. Therefore, by Lemma~\ref{lem:spec-mod}, $G^h(\aao,f_{1,m}^0,f_{2,m}^0,1)=0$ modulo $m^0$. Hence, $\mathcal{F}(\aao)$ is a $\K_{\aao}$--family of points of $\Cu(G^h,\aao)$. 

If $\mathcal{F}:=\{(L_{1,m}:L_{2,m}:0)\}_{m}$ all the arguments above apply and, hence, one deduces that $\mathcal{F}(\aao)$ is a $\K_{\aao}$--family of points of $\Cu(G^h,\aao)$.

It remains to prove that $\#(\mathcal{F}(\aao))=\#(\mathcal{F})$. Let $\mathcal{L}$ be the $\K$--linear change of coordinates transforming $\Cu(G)$ in regular position; see Step (1) in the standard decomposition process described in Subsection~\ref{subsec-standard-decomp}. Then, $\#(\mathcal{F})=\#(\mathcal{L}^{-1}(\mathcal{F}))$. Furthermore, $\mathcal{L}^{-1}(\mathcal{F})$ is in the form appearing either in~\eqref{eq-sing-infinity} or in~\eqref{eq-sing-affine}. Therefore, $\#(\mathcal{F})=\deg_t(m)$. Since $\mathcal{L}$ is over $\K$, we may apply it to $\mathcal{F}(\aao)$ and $\mathcal{L}^{-1}(\mathcal{F}(\aao))$ will be of the form either $\{(t:B^0:1)\}_{m^0}$ or $\{(1:t:0)\}_{m^0}$. In both cases, $\#(\mathcal{F}(\aao))=\#(\mathcal{L}^{-1}(\mathcal{F}(\aao)))=\deg_t(m^0)$. Now, the result follows using that $\deg_t(m)=\deg_t(m^0)$.
\end{proof}

\begin{remark}\label{rem-Fa}
Given an $\FF$--family $\mathcal{F}\in\Stand(G^h)$ (see~\eqref{eq-SingularLocus}), and $\aao\in \Omega_{\mathrm{def}(\mathcal{F})}$ (see Def.~\ref{def-defFam}), we observe that, even though $\mathcal{F}$ is irreducible, $\mathcal{F}(\aao)$ may be reducible. We are interested in working with irreducible specialized families. So, factoring over $\K_{\aao}$ the defining polynomial of $\mathcal{F}(\aao)$, the family will be decomposed as
\[ \mathcal{F}(\aao)=\bigcup_{i\in I} \mathcal{F}_i \]
where $\mathcal{F}_i$ is an irreducible $\K_{\aao}$--family. We refer to $\mathcal{F}_i$ as the \textit{irreducible subfamilies of $\mathcal{F}(\aao)$.}
Note that, if $\mathcal{F}(\aao):=\{(f_1:f_2:f_3)\}_{m}$, each irreducible subfamily is of the form 
$\{ (f_1:f_2:f_3)\}_{m^*}$ where $m^*(t)$ is an irreducible factor of $m(t)$ over  $\K_{\aao}$.
\end{remark}

In the sequel we analyze the multiplicity of families of singularities under specializations.

\begin{definition}\label{def-multFam} 
Let $\mathcal{F}:=\{(f_1:f_2:f_3)\}_{m}\in\Stand(G^h)$ (see~\eqref{eq-SingularLocus}) be an irreducible $\FF$--family of $r$-fold points of $\Cu(G^h)$ (see Def.~\ref{def-family-sing}), and let $H^*(\aa,x,y,z)$ be one of the order $r$ derivatives of $G^h$ such that $H^*(\aa,f_1,f_2,f_3)\neq 0$ modulo $m(\aa,t)$ (see Remark \ref{re-mult-fam}). Let $H(\aa,t)$ be the remainder of the division of $H^*(\aa,f_1,f_2,f_3)$ by $m(\aa,t)$ w.r.t. $t$. Let $R(\aa):=\res_t(H(\aa,t),m(\aa,t))$. We define the open subset
\[ \Omega_{\mathrm{mult}(\mathcal{F})}:=\Omega_{\mathrm{def}( \mathcal{F})}\cap  \Omega_{\mathrm{nonZ}(H)}  \cap \,\Omega_{\mathrm{nonZ}(R)}.
\]
\end{definition}

\begin{remark}
We observe that in Def.~\ref{def-multFam}, $m$ is irreducible over $\FF$, $H\in \FF[t]\setminus\{0\}$ and $\deg_t(H)<\deg_t(m)$. Therefore, $\gcd(m,H)=1$ and hence $R\neq 0$.
\end{remark}

\begin{lemma}\label{lem:r-fold}
Let $\mathcal{F}\in \Stand(G^h)$ (see~\eqref{eq-SingularLocus}) be an irreducible $\FF$-family of $r$-fold points of $\Cu(G^h)$.
If $\aao\in \Omega_{\mathrm{mult}(\mathcal{F})}$, then every irreducible subfamily of  $\mathcal{F}(\aao)$ (see Remark~\ref{rem-Fa}) is a $\K_{\aao}$--family of $r$-fold points of $\Cu(G^h,\aao)$.
\end{lemma}
\begin{proof}
$\mathcal{F}$ can be expressed as $\mathcal{F}=\{(f_1:f_2:\lambda)\}_{m}$ where $f_1,f_2,m\in \FF[t]$, $\lambda\in \{0,1\}$,
$m$ irreducible over $\FF$, and such that, for the case $\lambda=0$, $\deg_{t}(f_i)\leq 1$ and $\gcd(f_1,f_2)=1$ (see~\eqref{eq-SingularLocus}). Let $H^*, H$ be as in Def.~\ref{def-multFam}. Since $\aao\in \Omega_{\mathrm{def}(\mathcal{F})}$, by Lemma \ref{lem:defF}, $\mathcal{F}(\aao)$ is a $\K_{\aao}$--family of points of $\Cu(G^h,\aao)$. Let  $\mathcal{F}_i:=\{(f_1(\aao,t):f_2(\aao,t):\lambda)\}_{m^*}$  be an irreducible  subfamily of $\mathcal{F}(\aao)$. Now, let $W(\aa,x,y,z)$ be any partial derivative of $G^h$ of order smaller than $r$. Then, $W(\aa,f_1(\aa,t),f_2(\aa,t),\lambda)=0$ mod $m(\aa,t)$. Thus, since all the involved specializations are well--defined, $W(\aao,f_1(\aao,t),f_2(\aao,t),\lambda)=0$ mod $m^*$. Therefore, the points at $\fam_i$ have multiplicity at least $r$. On the other hand, since $\aao\in \Omega_{G}$, $H^*(\aao,f_1(\aao,t),f_2(\aao,t),\lambda)$ is well--defined and $\deg(m(\aa,t))=\deg(m(\aao,t))$, by Lemma~\ref{lem:spec-mod}, $H^*(\aao,f_1(\aao,t),f_2(\aao,t),\lambda)=H(\aao,t)$ modulo $m(\aao,t)$. Furthermore, since $\aao\in \Omega_{\mathrm{nonZ}(H)}$, then $H(\aao,t)\neq 0$. Moreover, since the leading coefficient of $m$ w.r.t. $t$ does not vanish at $\aao$, we have that
$\res_{t}(H(\aao,t),m(\aao,t))=\mu \,R(\aao)$ for some non-zero constant $\mu\in \overline{\K}$ (see Lemma 4.3.1 in~\cite{winkler}). So, since $\aao\in \Omega_{\mathrm{nonZ}(R)}$, $\res_{t}(H(\aao,t),m(\aao,t))\neq 0$. Therefore, $\gcd(H(\aao,t),m(\aao,t))=1$ and hence $H(\aao,t)\neq 0$ mod $m^*$.
Summarizing, $\mathrm{mult}(\mathcal{F}_i)=r$.
\end{proof}
 
In the last part of this subsection, we deal with the tangents to $\Cu(G^h)$ at an irreducible $\FF$-family. For this purpose, since the family $\mathcal{F}$ is irreducible, we will work with the curve $\CuFam(G^{h})$ associated to $\fam$ (see Def.~\ref{def-curve-fam} in Subsection~\ref{subsec-families} of the appendix).

\begin{definition}\label{def:tangent} Let $\mathcal{F}=\{(f_1:f_2:f_3)\}_m$ be an irreducible $\FF$-family of $r$-fold points of $\Cu(G^h)$.
Let $\FF_m$ be the quotient field of $\FF[t]/\!\!<\!\!m(t)\!\!>$.
The \textit{defining tangent polynomial} of $\Cu(G^h)$ at $\mathcal{F}$ is the homogenous polynomial $T\in\FF_m[x,y,z]$ of degree $r$ that defines the tangents, with its corresponding multiplicities, to $\CuFam(G^h)$ at the point $(f_1:f_2:f_3)$. Similarly, we introduce the \textit{defining tangent polynomial} to a specialized curve.
\end{definition}

\begin{remark}\label{rem-character} \
\begin{enumerate}
\item Let $\mathcal{F}=\{(f_1:f_2:1)\}_{m} \in \Stand(G^h)$ (see~\eqref{eq-SingularLocus}) be an irreducible family of affine $r$--fold points; similarly if the family is at infinity. The defining tangent polynomial $T$ is the reduction of
\begin{equation}\label{eq-Tstar}
T^*(\aa,t,x,y,z)=\sum_{i=0}^{r} \binom{r}{i} \dfrac{\partial^r G}{\partial^i x\,\partial^{r-i} y}(f_1,f_2) (x-f_1z)^i (y-f_2 z)^{r-i}
\end{equation}
modulo $m(t)$.
\item Let $\mathcal{F}$ be as in Def. \ref{def:tangent}. Then $\mathcal{F}$ is a family of ordinary points if and only if $T$ is squarefree over $\FF_m$. In the sequel, we assume w.l.o.g. that there is no tangent to $\CuFam(G^{h})$ at $\mathcal{F}$ independent of $x$ and that for two different tangents $T_1(\aa,x,y,z), T_2(\aa,x,y,z)$ it holds that $T_1(\aa,x,1,1)\neq T_2(\aa,x,1,1)$. Note that, if this is not the case, one can apply a linear change over $\K$ (and thus invariant under specializations of the parameters $\aa$). Then, the ordinary character of the family is readable from the squarefreeness of $T(\aa,t,x,1,1)$ over $\FF_m$.
\end{enumerate}
\end{remark}

\begin{definition}\label{def:ordSet}
Let $\mathcal{F}$, with defining polynomial $m(t)$, be an irreducible $\FF$--family of ordinary $r$-fold points of $\Cu(G^h)$. Let $T$ be the defining tangent polynomial of $\mathcal{F}$, where we assume  w.l.o.g. that the hypotheses in Remark~\ref{rem-character}~(2) are satisfied. Let $D(\aa,t)$ be the reduction modulo $m(t)$ of the discriminant w.r.t. $x$ of $T(\aa,t,x,1,1)$. Let $N(\aa)=\res_t(D,m)$. Let $A(\aa,t)$ be the leading coefficient of $T(\aa,t,x,1,1)$ w.r.t. $x$ and let $R(\aa)=\res_{t}(A,m)$. Let $S(\aa,x)=\res_t(T(\aa,t,x,0,0),m)$.
We define the set
\[ \Omega_{\mathrm{ord}(\mathcal{F})}=\Omega_{\mathrm{mult}(\mathcal{F})}\cap \Omega_{\mathrm{nonZ}(R)}\cap \Omega_{\mathrm{nonZ}(N)} \cap \Omega_{\mathrm{nonZ}(S)} . \]
\end{definition}

\begin{remark} In relation to Def.~\ref{def:ordSet}, we observe the following.
\begin{enumerate}
\item By construction, $\deg_t(A)<\deg_t(m)$ and clearly $A$ is not zero. Since $m$ is irreducible, $\gcd(A, m)=1$. Hence, $R\neq 0$.
\item Since $\mathcal{F}$ is ordinary and the two hypotheses in Remark \ref{rem-character}~(2) are satisfied, $D\neq 0$. 
Because $\deg_t(D)<\deg_t(m)$ and $m$ is irreducible, $\gcd(m,H)=1$ and hence, $N\neq 0$.
\item $T(\aa,t,x,y,z)$ has a factor in $\FF_{m}[y,z]$ if and only if $T(\aa,t,x,0,0)=0$.
This follows from the fact that the tangents are of degree one and thus, $T(\aa,t,x,y,z)=\prod (A_i(\aa,t)x+B_i(\aa,t)y+C_i(\aa,t)z)$ for some $A_i,B_i,C_i \in \FF_{m}$.
Thus, under our assumption that $T(\aa,t,x,y,z)$ does not have a factor independent of $x$, $T(\aa,t,x,0,0) \ne 0$.
Since $m$ is irreducible, and $\deg_t(T(\aa,t,x,0,0))<\deg_t(m)$, it follows that $S \ne 0$.
\end{enumerate}
\end{remark}

\begin{lemma}\label{lem:ord-fold}
Let $\mathcal{F}\in \Stand(G^h)$ (see \eqref{eq-SingularLocus}) be an irreducible $\FF$-family of ordinary $r$-fold points of $\Cu(G^h)$.
If $\aao\in \Omega_{\mathrm{ord}(\mathcal{F})}$, every irreducible subfamily of  $\mathcal{F}(\aao)$ (see Remark \ref{rem-Fa}) is a $\K_{\aao}$--family of ordinary $r$-fold points of $\Cu(G^h,\aao)$.
\end{lemma}
\begin{proof}
Since $\aao\in \Omega_{\mathrm{ord}(\mathcal{F})}\subset \Omega_{\mathrm{mult}(\mathcal{F})}\subset \Omega_{\mathrm{def}(\mathcal{F})}$,  $\deg_t(m(\aa,t))=\deg_{t}(m(\aao,t))$ and, by Lemma \ref{lem:r-fold}, every irreducible subfamily of $\mathcal{F}(\aao)$ is a $\K_{\aao}$--family of $r$-fold points of $\Cu(G^h,\aao)$. Let us prove that all points in $\mathcal{F}(\aao)$ are ordinary.

Let $T^*,T$  be as in Remark \ref{rem-character}  (1), 
 or similarly if the family is at infinity. Since $\aao\in \Omega_{\mathrm{ord}(\mathcal{F})}\subset \Omega_{\mathrm{mult}(\mathcal{F})}\subset \Omega_{\mathrm{def}(\mathcal{F})}\subset \Omega_G$, by Lemma \ref{lem:omegaG}, $T^*$ specializes properly at $\aao$, and since $\deg_t(m(\aa,t))=\deg_t(m(\aao,t))$, by Lemma \ref{lem:spec-mod}, $T$ also specializes properly at $\aao$.

Now, let $P$ be a point in $\mathcal{F}(\aao)$. Then, there exists a root $t_0$ of $m(\aao,t)$ such that $P$ is obtained by specializing $\mathcal{F}$ at $\aao$ and $t_0$. Since $P$ belongs to one of the irreducible subfamilies, $P$ is an $r$--fold point of the curve $\Cu(G^h,\aao)$. Because of the discussion above, $E(x,y,z):=T(\aao,t_0,x,y,z)$ is the defining tangent polynomial of $\Cu(G^h,\aao)$ at $P$. It remains to prove that $E$ is squareefree. First, let us see that there is no factor of $E$ independent of $x$. Assume that $e(y,z)$ is a factor of $E$. Then $E(x,0,0)=0$, and $(\aao,t_0,x,0,0)$ is a common zero of $T(\aa,t,x,y,z)$ and $m(\aa,t)$. Therefore, see e.g. Theorem 4.3.3 in~\cite{winkler}, $S(\aao,x)=0$ which contradicts that $\aao\in \Omega_{\mathrm{nonZ}(S)}$.
Thus, it is sufficient to prove the squarefreeness of $E(x,1,1)$. Let us assume that it is not squarefree, then its discriminant is zero. That is, the discriminant of $T(\aao,t_0,x,1,1)$ is zero. On the other hand, $\aao\in \Omega_{\mathrm{nonZ}(R)}$, $R(\aao)\neq 0$ and thus, $A(\aao,t_0)\neq 0$.
By~\cite[Lemma 4.1.3]{winkler} and the fact that $m(\aao,t_0)=0$, it follows that $D(\aao,t_0)=0$ and consequently, $N(\aao)=0$, in contradiction to $\aao \in \Omega_{\mathrm{nonZ}(N)}$.
\end{proof}

As a consequence of Lemmas~\ref{lem:defF}, \ref{lem:r-fold}, \ref{lem:ord-fold}, and taking into account that $\Omega_{\mathrm{ord}(\mathcal{F})}\subset\Omega_{\mathrm{mult}(\mathcal{F})}\subset \Omega_{\mathrm{def}(\mathcal{F})}$, we get the following corollary.

\begin{corollary}\label{cor:mul+ord}
Let $\mathcal{F}\in \Stand(G^h)$ (see~\eqref{eq-SingularLocus}) be an irreducible $\FF$-family of ordinary $r$-fold points of $\Cu(G^h)$.
If $\aao\in \Omega_{\mathrm{ord}(\mathcal{F})}$, then all points in $\mathcal{F}(\aao)$ are ordinary $r$-fold points of $\Cu(G^h,\aao)$ and $\#(\mathcal{F})=\#(\mathcal{F}(\aao))$.
\end{corollary}

\section{Preservation of the Genus}\label{sec-genus}
We consider a polynomial
\begin{equation}\label{eq-F} F(\mathbf{a},\gamma,x,y)\in \K[\aa,\gamma][x,y] \setminus \K[\aa,\gamma].
\end{equation}
$F$, as a non--constant polymomial in $\FF[x,y]$, defines an affine plane curve over $\overline{\FF}$ that we assume irreducible.

As introduced in Subsection~\ref{subsec-spec-families}, for each $\aao\in \S$ such that $F(\aao,\gamma^0,x,y)\not\in \overline{\K}$, we denote by $\Cu(F,\aao)$ the curve $\Cu(F(\aao,\gamma^0,x,y))$; similarly for $\Cu(F^h,\aao)$. Also, we denote by $\K_{\aao}$ the ground field of $\Cu(F,\aao)$ (see Def.~\ref{def-Ka}). Our goal is to analyze the relation between the genus of $\Cu(F)$ and the genus of $\Cu(F,\aao)$, under the assumption that $\Cu(F,\aao)$ is irreducible.

\subsection{Ordinary singular locus case}\label{subsec-ord-locus}
We start our analysis assuming that $\Cu(F^h)$ has only ordinary singularities.
Let $\Stand(F^h)$ be the $\FF$--standard decomposition of the singular locus of $\Cu(F^h)$ obtained by using the process described in Subsection~\ref{subsec-families}. Let $\Stand(F^h)$ decompose as in~\eqref{eq-SingularLocus} 
\begin{equation}\label{eq-SingularLocusF}
\Stand(F^h)=\bigcup_{m\in \mathcal{A}_a} \{(f_{1,m}:f_{2,m}:1)\}_{m} \, \cup \, \bigcup_{m\in \mathcal{A}_{\infty}} \{(L_{1,m}:L_{2,m}:0)\}_{m}
\end{equation}
where $f_{i,m}\in \FF[t]$, $L_{i,m}\in \K[t]$ with $\gcd(L_{1,m},L_{2,m})=1$ and $\deg(L_{i,m})\leq 1$, and $\mathcal{A}_a$, $\mathcal{A}_{\infty}$ are finite sets of irreducible polynomials in $\FF[t]$. We start with the following definition, where $\sing(\Cu(F^h))$ denotes the singular locus of $\Cu(F^h)$.

\begin{definition}\label{def-singOrd} We define the open subset (see~\eqref{eq-SingularLocusF} and Def.~\ref{def:omegaG} and~\ref{def:ordSet})
$$\Omega_{\mathrm{singOrd}(F^h)}:=\left\{ \begin{array}{cc} \displaystyle{\bigcap_{ \mathcal{F}\in \Stand(F^h)} \Omega_{\mathrm{ord}(\mathcal{F})}} & \text{if $\sing(\Cu(F^h))\neq \emptyset$}\\
\Omega_{F} & \text{if $\sing(\Cu(F^h))=\emptyset$} \end{array} \right.$$
\end{definition}

Then, the following result holds.

\begin{theorem}\label{theorem:genus-weak}
Let $\aao\in \Omega_{\mathrm{singOrd}(F^h)}$. If $\Cu(F,\aao)$ is irreducible, then
\[ \genus(\Cu(F))\geq \genus(\Cu(F,\aao)). \]
\end{theorem}
\begin{proof} Let $d$ be the degree of $\Cu(F)$. If $\sing(\Cu(F^h))=\emptyset$, then $\aao\in \Omega_F$. So, by Lemma~\ref{lem:omegaG}, $\deg(F(\aao,\gamma^0,x,y))=d$. Since $\Cu(F,\aao)$ is irreducible, by equality~\eqref{eq-genus} in the appendix, one has that
\[ \genus(\Cu(F))=\dfrac{(d-1)(d-2)}{2} \geq  \genus(\Cu(F,\aao)). \]
If $\sing(\Cu(F^h))\neq \emptyset$, then $\aao\in \Omega_{\mathrm{singOrd}(F^h)}\subset \Omega_F$ and $\deg(F(\aao,\gamma^0,x,y))=d$. By Corollary~\ref{cor:mul+ord}, all elements in $\sing(\Cu(F^h))$ have the same multiplicity and character as their corresponding elements in $\sing(\Cu(F^h,\aao))$ after specialization. 
New singularities, however, may appear in $\sing(\Cu(F^h,\aao))$. So, reasoning as above with the genus formula in~\eqref{eq-genus}, or~\eqref{eq-genus2}, in the appendix, we get the result.
\end{proof}

The next result is a direct consequence of the previous theorem when the genus of $\Cu(F)$ is zero.

\begin{corollary}\label{cor:genus-zero}
Let $\Cu(F)$ be a rational curve. Let $\aao\in \Omega_{\mathrm{SingOrd}(F^h)}$. If $\Cu(F,\aao)$ is irreducible, then $\Cu(F,\aao)$ is rational.
\end{corollary}

The inequality in Theorem~\ref{theorem:genus-weak} comes from the fact that, using $\Omega_{\mathrm{SingOrd}(F^h)}$, we cannot ensure that $\sing(\Cu(F^h,\aao))$ does not include new singularities apart from those coming from the specialization of the singular locus of $\Cu(F^h)$. To control this phenomenon, we will ensure that certain Gr\"obner bases behave properly under specializations. By, exercises 7, 8, pages 315--316 in~\cite{cox}, or by Proposition 1, page 308 in~\cite{cox}, we know that there exists an open Zariski set such that the Gr\"obner basis specializes properly; in fact, a description of this open subset is also available. For a more general analysis of Gr\"obner bases with parametric coefficients we refer to~\cite{montes} and~\cite{weispfennin}. On the other hand, since we are working with bivariate polynomials in $\FF[x,y]$, the open subset above can be determined by using resultants.
This motivates the next definition.

\begin{definition}\label{def-OpenGB}
Let $\mathrm{I}$ be an ideal in $\FF[\overline{v}]$, where $\overline{v}$ is tuple of variables, generated by $\mathscr{G}\subset \FF[\overline{v}]$. Let $\mathcal{G}$ be a Gr\"obner basis of $\mathscr{G}$ w.r.t. some order. We define   $\Omega_{\mathrm{spGB}(\mathcal{G})}\subset \S$ as a non-empty open subset such that for every $\aao\in \Omega_{\mathrm{spGB}(\mathcal{G})}$ it holds that
$\{ g(\aao,\gamma^0,\overline{v}) \,|\, g\in \mathcal{G}\}$ is a Gr\"obner basis, w.r.t. the same order, of the ideal generated by $\{ g(\aao,\gamma^0,\overline{v}) \,|\, g\in \mathscr{G}\}$ in $\K_{\aao}[\overline{v}]$.
\end{definition}

Now, we focus our attention on the standard decomposition of the singular locus of $\Cu(F^h)$ described in~\eqref{eq-SingularLocusF}.
In the first step, if necessary, we apply a $\K$ linear change of coordinates to ensure that the curve is in regular position. Hence, this linear transformation it is not affected by the specializations of $\aa$. Therefore, for our reasonings, we may assume w.l.o.g. that $F$ is already in regular position. Next, let $\mathcal{G}_1$ be a Gr\"obner basis of $<\!\!F,F_{x},F_{y}\!\!>$ w.r.t. the lexicographic order with $x<y$, and let $\mathcal{G}_2$ be a Gr\"obner basis of the same ideal w.r.t. the lexicographic order with $y<x$. Let $\{f(\aa,\gamma,x)\}=\mathcal{G}_1\cap \FF[x]$, $\{g(\aa,\gamma,y)\}=\mathcal{G}_2\cap \FF[y]$, $\tilde{f}=f/\gcd(f,f_x)$ and $\tilde{g}=g/\gcd(g,g_y)$, where $f_x$ is the derivative of $f $ w.r.t. $x$; similarly with $g_y$. Finally, let $\mathcal{G}_3:=\{A(\aa,x),y-B(\aa,x)\}$, with $A$ square-free and $\deg(B)<\deg(A)$, be the normed reduced Gr\"obner basis w.r.t. the lexicographic order with $x<y$ of $<\!\!F, F_{x}, F_{y}, \tilde{f}, \tilde{g}\!\!>$.
Then, we introduce the following definitions (recall  Def.~\ref{def-open1}, \ref{def-gcd}, \ref{def:sqfree},  \ref{def:gcd-several}, \ref{def:sqfree} and~\ref{def-singOrd}).

\begin{definition}\label{def:genusOrdAffine} \textsf{(Affine singularities)} 
With the notation introduced above, let \[ \Omega_1=\Omega_{\mathrm{nonZ}(U)} \cap \Omega_{\mathrm{gcd}(f,f_x)}\cap \Omega_{\mathrm{sqfree}(\tilde{f})}, \,\, \,\,\Omega_2=\Omega_{\mathrm{nonZ}(V)} \cap \Omega_{\mathrm{gcd}(g,g_y)}\cap \Omega_{\mathrm{sqfree}(\tilde{g})}
\]
where $U$ and $V$ are the leading coefficients of $f$ and $g$ w.r.t. $x$ and $y$, respectively.
We define the open subset
\[ \Omega_{\mathrm{sing}_a(F)}:=\bigcap_{i=1}^{3} \Omega_{\mathrm{spGB}(\mathcal{G}_i)}  {\bigcap_{q\in \mathcal{G}_1\setminus \{f\}}
\Omega_{\mathrm{nonZ}(W_{q,y})}} {\bigcap_{q\in \mathcal{G}_2\setminus \{g\}}
\Omega_{\mathrm{nonZ}(W_{q,x})}} \cap \Omega_1\cap \Omega_2 \cap \Omega_{\mathrm{sqfree}(A)}\]
where $W_{q,y}$ denotes the leading coefficient of $q$ w.r.t. $y$; similarly with $W_{q,x}$.
\end{definition}

\begin{remark}
Note that all polynomials in $\mathcal{G}_1\setminus \{f\}$ do depend on $y$; similarly for $\mathcal{G}_2\setminus \{g\}$. The idea of controlling the coefficients $W_{q,x}$ and $W_{q,y}$ in Def.~\ref{def:genusOrdAffine} is to ensure that the elimination ideal of the specialized Gr\"obner basis does not include additional generators.
\end{remark}

\begin{definition}\label{def:genusOrdInf} \textsf{(Singularities at infinity)} 
Let $M(\aa,\gamma,y,z):=F^h(\aa,\gamma,1,y,z)$, let $U(\aa,\gamma,t)=\gcd(M(\aa,\gamma,t,0),M_y(\aa,\gamma,t,0),M_z(\aa,\gamma,t,0))$ and $\tilde{U}(\aa,\gamma,t):=U/\gcd(U,U')$, where $U'$ is the derivative of $U$ w.r.t. $t$.  Let 
$$\Omega_{(0:1:0)}:=\left\{
\begin{array}{cl}
\S  & \text{if $(0:1:0)\in \sing(\Cu(F^h))$} \\
\Omega_{\mathrm{nonZ}(J(\aa,\gamma,0,1,0))}
& \text{if $(0:1:0)\notin \sing(\Cu(F^h))$} \end{array} \right.$$
where $J$ is one the first derivatives of $F^h$ not vanishing at $(0:1:0)$.  We define the open subset  
\[ \Omega_{\mathrm{sing}_\infty(F)}:=\Omega_{\mathrm{gcd}(M(\aa,\gamma,t,0),M_y(\aa,\gamma,t,0),M_z(\aa,\gamma,t,0))} \cap \Omega_{\mathrm{gcd}(U,U')} \cap \Omega_{(0:1:0)}. \]
\end{definition}
\begin{definition}\label{def:genusOrd}
We define the open subset 
\[ \Omega_{\mathrm{genusOrd}(F^h)}:= \Omega_{\mathrm{singOrd}(F^h)} \cap \Omega_{\mathrm{sing}_a(F)}\cap \Omega_{\mathrm{sing}_\infty(F)}. \]
\end{definition}

In the following lemma, we see that the cardinality of the singular locus, as a set, is preserved under specializations.
\begin{lemma}\label{lem:card-sing}
Let $\aao\in \Omega_{\mathrm{sing}_a(F)}\cap \Omega_{\mathrm{sing}_\infty(F)}$. Then
\[ \#(\sing(\Cu(F^h))=\#(\sing(\Cu(F^h,\aao)). \]
\end{lemma}
\begin{proof}
Let us denote the specialization of a polynomial $Q$ at $\aao$ as $Q^0$. 
Let $\mathcal{G}^{0}_{1}$ be a Gr\"obner basis of $<\!\!F^0,F^{0}_{x},F^{0}_{y}\!\!>$ w.r.t. the lexicographic order $x<y$, and let $\mathcal{G}^{0}_2$ be a Gr\"obner basis of the same ideal w.r.t. the lexicographic order $y<x$. Since $$\aao\in \Omega_{\mathrm{spGB}(\mathcal{G}_1)} \cap \bigcap_{q\in \mathcal{G}_1\setminus \{f\}}\Omega_{\mathrm{nonZ}(W_{q,y})}\cap \Omega_1,$$ then $\{f^0\}=\mathcal{G}_{1}^{0}\cap \K_{\aao}[x]$ and $\deg_x(f )=\deg_x(f^0)$. 
Similarly, $\{g^0\}=\mathcal{G}_{2}^{0}\cap \K_{\aao}[y]$ and   $\deg_y(f)=\deg_y(f^0)$.
Since $\aao \in \Omega_1$, by Corollary~\ref{cor-lemma-gcd}, $\gcd(f,f_x)^0=\gcd(f^0,f_{x}^{0})$ and $\deg_x(\gcd(f,f_x)^0)=\deg_x(\gcd(f,f_x))$.
Thus, $$\tilde{f}^0=\frac{f^0}{\gcd(f^0,f_{x}^{0})}.$$
By Lemma~\ref{lem:sqfree}, it holds that $\deg_x(\tilde{f})=\deg_x(\tilde{f}^0)$ and $\tilde{f}^0$ is squarefree. 
Similarly for $g$ and $\tilde{g}$ since $\aao\in \Omega_2$. In addition, by Lemma \ref{lem:omegaG}, 
\[ \sqrt{<\!\! F^0,F^{0}_{x},F^{0}_{y} \!>}=<\!\! F^0,F_{x}^0,F_{y}^0,\tilde{f}^0,\tilde{g}^0\!\!>.\]
Since $\aao\in \Omega_{\mathrm{spGB}(\mathcal{G}_3)}$, $\{A^0,y-B^0\}$ is a Gr\"obner basis of $\sqrt{<\!\!F^0,F^{0}_{x},F^{0}_{y} \!>}$.

Since $\aao\in \Omega_{\mathrm{sqfree}(A)}$, by Lemma~\ref{lem:sqfree}, $A^0$ is squarefree. Therefore, the number of affine singularities of $\Cu(F,\aao)$ is $\deg_x(A^0)$ and $\deg_x(A)$ is the number of affine singularities of $\Cu(F)$. By Lemma~\ref{lem:sqfree}, we get that $\deg_x(A^0)=\deg_x(A)$ and, hence, $\Cu(F)$ and $\Cu(F,\aao)$ have the same number of affine singularities.

It remains to prove that the number of singularities at infinity is also the same. First we observe
that, if $(0:1:0)\in \sing(\Cu(F^h))$, then $(0:1:0)\in \sing(\Cu(F^h,\aao))$. Moreover, since $\aao\in \Omega_{(0:1:0)}$, if $(0:1:0)\not\in \sing(\Cu(F^h))$, then $(0:1:0)\not\in \sing(\Cu(F^h,\aao))$. 
For the remaining singularities at infinity, denote by $\Sigma$ the set of the singularities of the form $(1:\mu:0)\in \sing(\Cu(F^h))$; similarly let $\Sigma^0$ be the set of singularities of this type in $\sing(\Cu(F^h,\aao))$. Let $M^0(y,z):=M(\aao,\gamma^0,y,z)$ (see Def.~\ref{def:genusOrdInf}), $U^0(t):=\gcd(M^0(t,0),M^{0}_{y}(t,0),M^{0}_{z}(t,0))$ and $\tilde{U}^0(t):=U^0/\gcd(U^0,(U^0)')$. Then,
\begin{equation}\label{eq-uo}
\#(\Sigma)=\deg_t(\tilde{U}) \,\,\, \text{and}\,\,
\#(\Sigma^0)=\deg_t(\tilde{U}^0).
\end{equation}
Since $\aao\in \Omega_{\mathrm{gcd}(M(\aa,\gamma,t,0),M_y(\aa,\gamma,t,0),M_z(\aa,\gamma,t,0))}$, by Theorem~\ref{theorem-lemma-gcd-several-pol}, it holds that
\begin{equation}\label{eq-u1}
U^0(t)=U(\aao,\gamma^0,t) \,\,\, \text{and}\,\, \deg_t(U^0(t))=\deg_t(U(\aa,\gamma,t)).
\end{equation}
Let $D(\aa,\gamma,t)=\gcd(U,U')$ and $D^0(t)=\gcd(U^0,(U^0)')$. Since $\aao\in \Omega_{\mathrm{gcd}(U,U')}$, by Corollary~\ref{cor-lemma-gcd}, one has that
\begin{equation}\label{eq-u2}
D^0(t)=D(\aao,\gamma^0,t) \,\,\, \text{and}\,\, \deg_t(D^0(t))=\deg_t(D(\aa,\gamma,t)).
\end{equation}
Now, by~\eqref{eq-uo}, \eqref{eq-u1} and \eqref{eq-u2}, we get that $\#(\Sigma)=\#(\Sigma^0)$ and this concludes the proof.
\end{proof}

\begin{theorem}\label{theorem:genus-strong}
Let $\aao\in \Omega_{\mathrm{genusOrd}(F^h)}$. If $\Cu(F,\aao)$ is irreducible, then
\[ \genus(\Cu(F))= \genus(\Cu(F,\aao)). \]
\end{theorem}
\begin{proof}
Since $\aao\in \Omega_{\mathrm{sing}_a(F)}\cap \Omega_{\mathrm{sing}_\infty(F)}$, by Lemma~\ref{lem:card-sing}, it holds that $\#(\sing(\Cu(F^h))=\#(\sing(\Cu(F^h,\aao)).$ On the other hand,
since $\aao\in \Omega_{\mathrm{SingOrd}(F^h)}$, by Corollary~\ref{cor:mul+ord}, we know that each ordinary $r$--fold in $\sing(\Cu(F^h))$ generates an ordinary $r$-fold in $\sing(\Cu(F^h,\aao))$.
Therefore, applying the genus formula~\eqref{eq-genus} in the appendix, we conclude the proof.
\end{proof}

\subsection{General case}
Let $F$ be as in~\eqref{eq-F}. But now, differently to the case of Subsection~\ref{subsec-ord-locus}, we do not introduce any assumption on the singular locus of the irreducible curve $\Cu(F^h)$. The key of our analysis is to reduce the general case to the case studied in Subsection~\ref{subsec-ord-locus}. For this purpose, we recall that any irreducible curve is birationally equivalent to a curve having only ordinary singularities; see e.g.~\cite[Theorem 7.4.]{walker}, \cite[Theorem 9.2.4]{mauro} or~\cite[Section 3.2.]{myBook} for a more computational description. This transformation, say $\varphi$, can be seen as a finite sequence of blowups of the irreducible families of non-ordinary singularities and, hence, as a finite sequence of compositions of quadratic Cremona transformations and linear transformations.  We conclude by the fact that the genus is invariant under birational transformations. 

Now, our goal is to find an open subset $\Omega_{\mathrm{blowup}}$ of $\S$ such that, when $\aa$ is specialized in $\Omega_{\mathrm{blowup}}$, the birationality of $\varphi$ is preserved. For this purpose, let $F_0(x,y)=F(x,y)$ and let $\mathcal{F}(F_0^h)$ be an irreducible $\FF$-family of non-ordinary singularities of $\Cu(F_{0}^{h})$ with defining polynomial $m_1(t_1) \in \FF[t_1]$. Let $\FF_{m_1}$ be the quotient field of $\FF[t_1]/\!\!<\!m_1(t_1)\!>$. 
Then we apply a linear transformation $\mathcal{L}_1$, given by a matrix $M_1 \in \text{M}_{3 \times 3}(\FF_{m_1})$, and the Cremona transformation $\Qa_1=(yz:xz:xy)$ as described in the blow up basic step in Subsection~\ref{subsec-genus} of the appendix.
Let $\Delta_1:=\det(M_1)$ and let $\Cu(F_1^h)$ be the curve over $\overline{\FF_{m_1}}$ obtained after the quadratic transformation $\Qa_1 \circ \mathcal{L}_1$. Note that $F_{1}^{h}$, the quadratic transformation of $F_0^h$, is the cofactor of $F_{0}^h(\mathcal{Q}_1(\mathcal{L}_1))$ not being divisible by neither $x$, $y$ nor $z$.
We repeat the above process for $F_1^h(t_1,x,y,z)$, $F_2^h(t_1,t_2,x,y,z), \ldots, F_r^h(t_1,\ldots,t_r,x,y,z)$ until all singularities of $\Cu(F_r^h)$ are ordinary.
Then \[ \varphi(\aa,\gamma,t_1,\ldots,t_r,x,y,z) = (\Qa_r \circ \mathcal{L}_r) \circ \cdots \circ (\Qa_1 \circ \mathcal{L}_1) .\]
Note that $F_r^h$ is defined over $\FF(t_1,\ldots,t_r)$ and $\Cu(F_r^h)$ over the algebraic closure of $\FF(t_1,\ldots,t_r)$.
In addition, $\FF(t_1,\ldots,t_r)=\K(\aa,\gamma,t_1,\ldots,t_r)=\LL(\gamma,t_1,\ldots,t_r)$. So, we consider a primitive element of the extension over $\LL$, say $\gamma^*$, and we work over $\LL(\gamma^*)=\LL(\gamma,t_1,\ldots,t_r)$; note that results in Section \ref{sec-specialization} apply to this new frame. In this situation, let us denote by $\Delta=\Delta_1 \cdots \Delta_r \in \LL(\gamma^*)$ the product of the determinants of the linear transformations $\mathcal{L}_1,\ldots,\mathcal{L}_m$. In addition, let
$$\mathcal{M}:=\{ \text{all entries of $M_i\in \mathrm{M}_{3\times 3}(\LL(\gamma^*))\}_{i\in \{1,\ldots,r\}}$}.$$
and let
$$
\begin{array}{lc}
\mathcal{B}:=&\{ F_{i}^h(\mathcal{Q}_{i+1}(\mathcal{L}_{i+1}))(\aa,\gamma^*,0,y,z), F_{i}^h(\mathcal{Q}_{i+1}(\mathcal{L}_{i+1}))(\aa,\gamma^*,x,0,z), \\
\noalign{\vspace*{1mm}}
&  F_{i}^h(\mathcal{Q}_{i+1}(\mathcal{L}_{i+1}))(\aa,\gamma^*,x,y,0)\}_{i\in \{0,\ldots,r-1\}}.
\end{array}$$

\begin{definition}\label{def-openBlup}
With the notation introduced above, we define the set $$\Omega_{\mathrm{blowup}(F)}:= \Omega_F \cap \bigcap_{h\in \mathcal{M}} \Omega_{\mathrm{def}(h)} \cap \Omega_{\mathrm{nonZ}({\Delta})} \cap \bigcap_{h\in \mathcal{B}} \Omega_{\mathrm{nonZ}(h)}.$$
\end{definition}

The previous observations lead to the following result.

\begin{lemma}\label{lem:specializationBirational}
Let $\aao \in \Omega_{\mathrm{blowup}}(F)$. Then $\varphi(\aao,(\gamma^*)^0,x,y,z)$ is birational.
\end{lemma}
\begin{proof}
Since $\aao \in \Omega_{\mathrm{blowup}}(F) \subset \bigcap_{h\in \mathcal{M}} \Omega_{\mathrm{def}(h)} \cap \Omega_{\mathrm{nonZ}({\Delta})}$, all $\mathcal{Q}_i \circ \mathcal{L}_i$ are well--defined, and birational, when $\aa$ is specialized as $\aao$. So $\varphi(\aao,(\gamma^*)^0,x,y,z)$ is birational.
\end{proof}

\begin{lemma}\label{lemm:biratR}
Let $\aao \in \Omega_{\mathrm{blowup}}(F)$ and let  $\varphi^0:=\varphi(\aao,(\gamma^*)^0,x,y,z)$. If $F_0(\aao,\gamma^0,x,y,z)$ is irreducible, then $F_r(\aao,(\gamma^*)^0,x,y,z)$ is the quadratic transformation of $F_0(\aao,\gamma^0,x,y,z)$.
\end{lemma}
\begin{proof}
First we observe that because of (the proof of) Lemma~\ref{lem:specializationBirational} each $\varphi_i:=\mathcal{Q}_i \circ \mathcal{L}_i$ is well defined at $\aao$ and it is birational. Let us prove the result by induction. By hypothesis $F_{0}^{h}(\aao,\gamma^0,x,y,z)$ is irreducible. We have the equality $F_{0}^{h}(\varphi_i)=x^{n_1}y^{n_2}z^{n_3} F_{1}^{h}$ for some $n_i\in \mathbb{N}$ and that neither $x$, $y$ nor $z$ divides $F_{1}^{h}$. So, $F_{0}(\varphi_i)(\aao,(\gamma^*)^0,x,y,z)=x^{n_1}y^{n_2}z^{n_3} F_{1}^{h}(\aao,(\gamma^*)^0,x,y,z)$. 
Moreover, since $\aao\in \bigcap_{h\in \mathcal{B}} \Omega_{\mathrm{nonZ}(h)}$, we know that neither $x$, $y$ nor $z$ divides $F_1(\aao,(\gamma^*)^0,x,y,z)$. Furthermore, since $\varphi_i$ is birational when specialized at $\aao$ and $F_{0}^{h}(\aao,\gamma^0,x,y,z)$ is irreducible, we have that $F_{1}^{h}(\aao,(\gamma^*)^0,x,y,z)$ is also irreducible.
Thus, $F_{1}^{h}(\aao,(\gamma^*)^0,x,y,z)$ is the quadratic transformation of $F_{0}^{h}$. Now, the $i$-induction step is reasoned analogously using that, by induction, $F_{i}^{h}(\aao,(\gamma^*)^0,x,y,z)$ is irreducible.
\end{proof}

With this, we can now give an open set where the genus is preserved.

\begin{definition}\label{def-genus}
Let $F \in \FF[x,y]$ be as in~\eqref{eq-F}, and let $G \in \FF(\gamma^*)[x,y]$ be the polynomial obtained after the blowup process of $F^h$. We define the set $$\Omega_{\mathrm{genus}(F)}:= \Omega_{\mathrm{blowup}(F)} \cap \Omega_{\mathrm{genusOrd}(G^h)}.$$
\end{definition}

\begin{theorem}\label{thm:genusPreservation}
Let $F \in \FF[x,y]$ be as in~\eqref{eq-F}, and let $\aao \in \Omega_{\mathrm{genus}(F)}$. If $\Cu(F,\aao)$ is irreducible, then
\[ \genus(\Cu(F))= \genus(\Cu(F,\aao)). \]
\end{theorem}
\begin{proof}
We start recalling that the quadratic transformation defines a birational map (see e.g.~\cite[Chapter V, Section 4.2]{walker}) from the original curve on its quadratic transformed curve (see the precise definition of transformed curve in Step 2 of the blow up at a point in Subsection~\ref{subsec-genus} in the appendix), and that the genus is a birational invariant (see e.g.~\cite[Theorem 7.2.2 or Prop. 8.5.1]{mauro}). Then, since $G^h$ is the quadratic transformation of $F^h$, we have that
\begin{equation}\label{eqq-genus1}
\genus(\Cu(F^h(\aa,x,y,z)))=\genus(\Cu(G^h(\aa,\gamma^*,x,y,z))).
\end{equation}

Let $\varphi^0$ denote the map $\varphi(\aao,(\gamma^*)^0,x,y,z)$. Since $\aao \in \Omega_{\mathrm{blowup}(F)}$, by Lemma~\ref{lem:specializationBirational}, $\varphi^0$ is birational and, by Lemma~\ref{lemm:biratR}, $G^h(\aao,(\gamma^*)^0,x,y,z)$ is the quadratic transformation of $F^h(\aao,\gamma^0,x,y,z)$ via $\varphi^0$. Therefore, using the same argumentation as in~\eqref{eqq-genus1}, on gets
\begin{equation}\label{eqq-genus2}
\genus(\Cu(F^h(\aao,\gamma^0,x,y,z)))=\genus(\Cu(G^h(\aao,(\gamma^*)^0,x,y,z))).
\end{equation}
Moreover, since $\aao \in \Omega_{\mathrm{genusOrd}(G^h)}$, by Theorem~\ref{theorem:genus-strong}, it holds that
\begin{equation}\label{eqq-genus3}
\genus(\Cu(G^h(\aa,\gamma^*x,y,z))) = \genus(\Cu(G^h(\aao,(\gamma^*)^0,x,y,z))) .
\end{equation}
Now, the proof follows from~\eqref{eqq-genus1}, \eqref{eqq-genus2} and~\eqref{eqq-genus3}.
\end{proof}

\section{Birational Parametrization of Parametric Rational Curves}\label{sec-parametricCurves}
In Section~\ref{sec-genus}, and more precisely in Theorems~\ref{theorem:genus-strong} and~\ref{thm:genusPreservation}, we have described open subsets of $\S$ where the genus of the curve is preserved under specializations; even in Corollary~\ref{cor:genus-zero} the particular case of genus zero was treated. Nevertheless, in all these results the additional condition on the irreducibility, over $\overline{\K}$, of the specialized polynomial was required. Avoiding the irreducibility is in general a difficult problem related to the Hilbert irreducibility problem (see e.g.~\cite{Serre}). More precisely, there is no algorithm known that finds for given irreducible $F \in \FF[x,y]$ the specializations $\aao \in \S$ such that $F(\aao,x,y)$ is reducible. 
Nevertheless, in this section, we will analyze the particular case where $F$ defines a genus zero curve. In this case, we will introduce an open subset where the irreducibility is guaranteed.

For this purpose, throughout this section, let assume that $F \in \FF[x,y]$ is as in~\eqref{eq-F} and additionally assume that $\Cu(F)$ is rational. Moreover, let us assume that $\cP$ is a proper (i.e. birational) parametrization of $\Cu(F)$ which can be computed, for instance, by the algorithm described in Subsection~\ref{subsec:param} in the appendix. Note that, in general, one may need to extend $\FF$ with an algebraic element $\delta$ of degree two (see Subsection~\ref{sec-FieldOfparametrization}). 
If $\#(\mathbf{a})=1$ and $\deg(\gamma)=1$, or $\deg_{x,y}(F)$ is odd, then no extension of $\FF$ is required (see Remark~\ref{rem:TsenGeneralization} and Theorem~\ref{theorem:HHext}).

So, we may consider a primitive element of $\LL(\gamma,\delta)$, say $\gamma^*$, and express our parametrization in $\LL(\gamma^*)$. Throughout this section, by abuse of notation, let $\gamma$ play the role of $\gamma^*$ and let $\FF$ denote the field $\LL(\gamma^*)$. Let us write the proper parametrization $\cP$ of $\Cu(F)$ as
\begin{equation}\label{eq-Pa}  
\Pa(\aa,\gamma,t)=\left(\dfrac{p_1}{q_1},\dfrac{p_2}{q_2} \right) \in \FF(t)^2 \setminus \FF^2
\end{equation}
where we assume that $\Pa$ is in reduced form, that is $\gcd(p_1,q_1)=\gcd(p_2,q_2)=1$.

Let us start with the simple case of degree one.

\begin{definition}\label{def:deg1}
Let $F=A_2(\aa,\gamma)x+A_1(\aa,\gamma) y+A_0(\aa,\gamma)\in \FF[x,y] \setminus \FF$. Since $\Cu(F)$ is a line, it can be properly parametrized by a polynomial parametrization of the form $\Pa(\aa,\gamma,t)=(\lambda_1 t+\lambda_0, \mu_1 t+ \mu_0)\in \FF(t)^2\setminus \FF^2$. We define the set (see Def.~\ref{def-open1} and~\ref{def:omegaG})
\[ \Omega_{\mathrm{proper}(\Pa)}:= \Omega_{F} \cap \Omega_{\mathrm{def}(\lambda_1 t+\lambda_0)} \cap \Omega_{\mathrm{def}(\mu_1 t+\mu_0)} \cap \left( \Omega_{\mathrm{nonZ}(\lambda_1)} \cup \Omega_{\mathrm{nonZ}(\mu_1)}\right).\]
\end{definition}

\begin{proposition}\label{prop:deg1}
Let $F$ and $\cP$ be as in Def.~\ref{def:deg1}. For every $\aao\in \Omega_{\mathrm{proper}(\Pa)}$, it holds that $\Pa(\aao,t)$ is a proper polynomial parametrization of $\Cu(F,\aao)$.
\end{proposition}
\begin{proof}
Since $\aao\in \Omega_{F}$, by Lemma~\ref{lem:omegaG}, $F(\aao,\gamma^0,x,y)$ is well defined and $\Cu(F,\aao)$ is an affine line, obviously irreducible. Since $\aao\in \Omega_{\mathrm{def}(\lambda_1 t+\lambda_0)} \cap \Omega_{\mathrm{def}(\mu_1 t+\mu_0)}$ then $\Pa(\aao,\gamma^0,t)$ is well--defined. Moreover, since $\aao\in  \left( \Omega_{\mathrm{nonZ}(\lambda_1)} \cup \Omega_{\mathrm{nonZ}(\mu_1)}\right)$, $\Pa(\aao,\gamma^0,t)$ is a polynomial parametrization. Furthermore, $F(\aao,\gamma^0,\Pa(\aao,\gamma^0,t))=0$. So, it parametrizes the line $\Cu(F,\aao)$. In addition, since the polynomials in the parametrization are linear, the parametrization is invertible on the line $\Cu(F,\aao)$ and thus proper. Let $\mathcal{D}$ be the irreducible curve parametrized by $\Pa(\aao,\gamma^0,t)$ (that is the Zariski closure of the image of $\overline{\K}$ by  $\Pa(\aao,\gamma^0,t)$). Then, since $F(\aao,\gamma^0,\Pa(\aao,\gamma^0,t))=0$, $\Cu(F,\aao)$ and $\mathcal{D}$ have infinitely many intersection points. 
So, since they are irreducible, by B\'ezout Theorem (see e.g.~\cite{walker}) $\Cu(F,\aao)=\mathcal{D}$ and we conclude that $\Pa(\aao,\gamma^0,t)$ is a proper parametrization of $\Cu(F,\aao)$.
\end{proof}

\begin{remark}\label{rem-caseslineal}
With the notation in Proposition~\ref{prop:deg1}, if $\aao\in \S\setminus  \Omega_{\mathrm{proper}(\Pa)}$, then:
\begin{enumerate}
\item If $\aao\not\in \Omega_F\setminus  \Omega_{\mathrm{def}(F)}$, it holds that $ F(\aao,\gamma^0,x,y)=A_0(\aao,\gamma^0)\in \overline{\K}$, and hence $\Cu(F,\aao)$ does not define an affine curve.
\item If $\aao\not\in \Omega_{\mathrm{def}(\lambda_1 t+\lambda_0)}$ but $\aao\in \Omega_F$ (similarly if $\aao\not\in\Omega_{\mathrm{def}(\mu_1 t+\mu_0)}$), the specialization $\cP(\aao,\gamma^0,t)$ is not well--defined even though $\Cu(F,\aao)$ is a line. In this case, $\Cu(F,\aao)$ is rational and a proper parametrization of the specialized line can be provided.
\item If $\aao\not\in \left( \Omega_{\mathrm{nonZ}(\lambda_1)} \cup \Omega_{\mathrm{nonZ}(\mu_1)}\right)$ but $\aao\in \Omega_F\cap \Omega_{\mathrm{def}(\lambda_1 t+\lambda_0)} \cap \Omega_{\mathrm{def}(\mu_1 t+\mu_0)}$, then $\cP(\aao,\gamma^0,t)\in \overline{\K}^2$ and hence it is not a parametrization although $\Cu(F,\aao)$ is a line. As in case (2), a polynomial parametrization of the specialized line can be provided.
\end{enumerate}
\end{remark}

In the sequel, we assume that $\Cu(F)$ is not a line. We generalize the open subset in Def.~\ref{def:deg1} as follows.

\begin{definition}\label{def-omega}\
\begin{enumerate}
\item Let $\Omega_1:=\Omega_{F}$ (see Def. \ref{def:omegaG}).
\item $\Omega_2:=\Omega_{\mathrm{def}(p_1)}\cap \Omega_{\mathrm{def}(p_2)}\cap \Omega_{\mathrm{nonZ}(q_1)}\cap \Omega_{\mathrm{nonZ}(q_2)}$ (see \eqref{eq-Pa}).
\item We consider the polynomials $G_i=p_i(h)q_i(t)-p_i(t)q_i(h)\in \FF[h][t]\setminus \{0\}$ for $i\in \{1,2\}$. Let $\Omega_{3}:=\Omega_{\mathrm{gcd}(G_1,G_2)}$ (see Def. \ref{def-gcd}).
\item Let $\Omega_{4}:=\Omega_{\mathrm{gcd}(p_1,q_1)} \cap \Omega_{\mathrm{gcd}(p_2,q_2)}$; note that $p_i,q_j\in \FF[t]\subset \FF[h,t]$ and, since $\Cu(F)$ is not a line, the $p_i$ and $q_i$ are  non--zero (see Def. \ref{def-gcd}).
\end{enumerate}
We define $\Omega_{\mathrm{proper}(\Pa)}$ as
\[ \Omega_{\mathrm{proper}(\Pa)}= \bigcap_{i=1}^{4} \Omega_i. \]
\end{definition}

The following theorem generalizes Prop.~\ref{prop:deg1}.

\begin{theorem}\label{thrm:genus0general}
Let $\aao\in \Omega_{\mathrm{proper}(\Pa)}$. Then $\Cu(F,\aao)$ is a rational affine curve in $\overline{\K}^2$ properly parametrized by $\Pa(\aao,\gamma^0,t)$.
\end{theorem}
\begin{proof} 
If $\deg(F)=1$, the result follows from Prop.~\ref{prop:deg1}. Let $\deg(F)>1$. Since $\aao\in \Omega_1$, then  $F^0:=F(\aao,\gamma^0,x,y)$ is well--defined and $\deg(\Cu(F))=\deg(\Cu(F,\aao))$. In particular $\Cu(F,\aao)$ is an affine curve. On the other hand, since $\aao\in \Omega_2$, by Lemma~\ref{lem:defVan}~(1), we have that  $\Pa^0:=\Pa(\aao,\gamma^0,t)$ is well--defined. In addition, since $\aao\in \Omega_4$, the leading coefficients of $p_1,p_2,q_1,q_2$ do not vanish at $\aao$ (see Def.~\ref{def-gcd}). Consequently the degree of all numerators and denominators of $\Pa$ after specialization are preserved.
Furthermore, by Corollary~\ref{cor-lemma-gcd}, and using that $p_i/q_i$ are in reduced form, we get that $p_{i}^{0}/q_{i}^{0}$ are also in reduced form. Therefore,
\begin{equation}\label{eq-deg}
\deg_t\left(\dfrac{p_{i}^{0}}{q_{i}^{0}}\right)= \deg_t\left(\dfrac{p_i}{q_i}\right) \,\, \text{for $i\in\{1,2\}$.}
\end{equation}
In particular, $\cP^0\not\in \overline{\K}^2$, and hence $\cP^0$ is a parametrization.
Moreover, $F^0(\cP^0)=0$. Thus, $\Pa^0$ parametrizes the curve defined by one of the factors, say $H(x,y)$, of $F^0$. Let us see that indeed $\Cu(F,\aao)=\Cu(H)$.

Let $G_i(\aa,\gamma,h,t)$ as in Def.~\ref{def-omega}~(3), and let  $G:=\gcd(G_1(\aa,\gamma,h,t),G_2(\aa,\gamma,h,t))$. Let $\tilde{G}_i(h,t)$ be the corresponding polynomials associated, as in Def.~\ref{def-omega}~(3), to  $\Pa^0$. Let $\tilde{G}:=\gcd(\tilde{G}_1,\tilde{G}_2)$. Since $p_{i}^{0}/q_{i}^{0}$ are in reduced form, no simplification of the rational functions have been required, and therefore $\tilde{G}_i(h,t)=G_i(\aao,\gamma^0,h,t)$. Moreover, since $\aao\in \Omega_3$, by Lemma~\ref{lem-gcd}, it holds that
\begin{equation}\label{eq:gcd}
\deg_t(G(\aao,\gamma^0,h,t))=\deg_t(\tilde{G}(h,t)),
\end{equation}
and $\deg_t(G(\aao,\gamma^0,h,t))=\deg_t(G(\aa,\gamma,h,t))$. By~\cite[Theorem 3]{SW01}, since $\Pa(\aa,\gamma,t)$ is proper, we have that $\deg_t(G(\aa,\gamma,h,t))=1$. Therefore, it holds that $\deg_t(\tilde{G}(h,t))=\deg_t(G(\aao,\gamma^0,h,t))=1$. Again by~\cite[Theorem 3]{SW01}, $\Pa^0$ is proper.
On the other hand, by~\eqref{eq-deg}, $\deg_t(\Pa)=\deg_t(\Pa^0)$. Therefore, by Theorem 4.21 in~\cite{myBook}, we have that
\[ \max\{\deg_{x}(F),\deg_{y}(F)\}=\deg_t(\Pa)
= \deg_t(\Pa^0)\\
=  \max\{\deg_{x}(H),\deg_{y}(H)\}.
 \]
Moreover, since $F$ is not linear, by~\eqref{eq-deg}, no component of $\Pa^0$ is constant. Applying again~\cite[Theorem 4.21.]{myBook}, we have that
\[   \deg_{x}(H)=\deg_t\left(\dfrac{p_2^0}{q_2^0}\right)=\deg_t\left(\dfrac{p_2}{q_2}\right)=\deg_{x}(F), \,\,\,\deg_{y}(H)=\deg_t\left(\dfrac{p_1^0}{q_1^0}\right)=\deg_t\left(\dfrac{p_1}{q_1}\right)=\deg_{y}(F).  \]
Finally, since $H(x,y)$ divides $F^0$, one
has that $\Cu(F,\aao)=\Cu(H)$, which concludes the proof.
\end{proof}

\begin{remark}\label{cases} Let us analyze the behavior of $F$ and/or $\cP$ when specializing in $\S\setminus \Omega_{\mathrm{proper}(\cP)}$.
\begin{enumerate}
\item If $\aao\in \S\setminus \Omega_{1}$, since $F\in \K[\aa][x,y]$ (see \eqref{eq-F}), then $F(\aao,x,y)$ is always well--defined, and hence, $\deg(F(\aao,x,y))< \deg(F(\aa,x,y))$. So, it can happen that either $0<\deg(F(\aao,x,y))< \deg(F(\aa,x,y))$, so $\Cu(F,\aao)$ is an affine curve; or $0=\deg(F(\aao,x,y))$, which implies that $\Cu(F,\aao)$ is the empty set or $\overline{\K}^2$.
\item If $\aao\in \S\setminus \Omega_{2}$, then $\Pa(\aao,t)$ is not defined, and hence the specialization fails.
\item If $\aao\in \left(\S\setminus (\Omega_{3}\cap \Omega_{4})\right)\cap \Omega_1\cap \Omega_2$, at least one of the following assertions hold.
\begin{enumerate}
\item $\cP(\aao,t)$ is not proper.
\item $\cP(\aao,t)\in \overline{\K}^2$ and hence, $\cP(\aao,t)$ is not a parametrization.
\item$\cP(\aao,t)$ parametrizes a proper factor of $F(\aao,x,y)$, that is, $\Cu(F,\aao)$ decomposes and one of its components is rational and parametrized by $\cP(\aao,t)$.
\end{enumerate}
\end{enumerate}
\end{remark}

The next result follows from Theorem~\ref{thrm:genus0general} and emphasizes the polynomiality of the parametrizaion.

\begin{corollary}
If $\cP$ is proper and polynomial and $\aao\in \Omega_{\mathrm{proper}(\cP)}$, then $\cP(\aao,\gamma^0,t)$ parametrizes properly and polynomially $\Cu(F,\aao)$.
\end{corollary}

We now analyze the normality (i.e. the surjectivity, see~\cite{sendraNormal},~\cite{myBook}) of the parametrization. We recall that any parametrization can be reparametrized surjectively (see~\cite[Theorem 6.26]{myBook}). This reparametrization requires, in our case, a new algebraic extension of $\FF$ via a new algebraic element. Alternatively, one may reparametrize normally the specialized parametrizations. In the following we deal with the case where $\cP$ is already normal and we want to preserve this property through the specializations.
For this purpose, we first introduce a new definition. We recall that for an affine plane parametrization $\cP$ of a curve $\Cu$, $\Cu\setminus \cP(\overline{\K})$ contains at most one point. That point, if it exists, is called the \textit{critical point} of the parametrization (see~\cite[Def. 6.24]{myBook}).

\begin{definition}
Let $\cP$ be as in~\eqref{eq-Pa}. If $\cP$ is normal, we define the set
\[  \Omega_{\mathrm{normal}(\cP)}:=\left\{
\begin{array}{cl}
\S & \text{if $\deg_t(p_1)>\deg_t(q_1)$ or $\deg_t(p_2)>\deg_t(q_2)$,} \\
\Omega_{\mathrm{gcd}(N_1,N_2)} &  \text{if $\deg_t(p_1)\leq \deg_t(q_1)$ and $\deg_t(p_2)\leq \deg_t(q_2)$}
\end{array}
\right.
\]
where $(\alpha_1/\beta_1, \alpha_2/\beta_2)\in \FF^2$ is the critical point of $\cP$  and, for $i\in \{1,2\}$, $N_i= \alpha_{i} q_i -\beta_{i} p_i$,
\end{definition}

\begin{corollary}
Let $\cP$ be proper and normal. For $\aao\in \Omega_{\mathrm{proper}(\cP)}\cap \Omega_{\mathrm{normal}(\cP)}$, $\cP(\aao,\gamma^0,t)$ parametrizes properly and normally $\Cu(F,\aao)$.
\end{corollary}
\begin{proof} 
As in the proof of Theorem~\ref{thrm:genus0general}, we use the same criterion for simplifying the notation. Namely, if $Q$ is a polynomial, or a rational function, we will denote by $Q^0$ its specialization at $\aao$. Similarly $\Pa^0$ will denote the specialization at $\aao$ of the parametrization $\Pa$.
In the proof of Theorem~\ref{thrm:genus0general} we have seen that, for $\aao\in \Omega_{\mathrm{proper}(\cP)}$, $\deg_t(p_i )=\deg_t(p_i^0)$, similarly for $q_i$, and that the rational functions in $\cP^0$ are in reduced form.
Now, if $\deg_t(p_1)>\deg_t(q_1)$ or $\deg_t(p_2)>\deg_t(q_2)$, the result follows from Theorem \ref{thrm:genus0general} and~\cite[Theorem 6.22]{myBook}. If $\deg_t(p_1)\leq \deg_t(q_1)$ and $\deg_t(p_2)\leq \deg_t(q_2)$, since $\aao\in \Omega_2$ in Def.~\ref{def-omega}, we get that $$C:=\left(\frac{\alpha_1^0}{\beta_1^0}, \frac{\alpha_2^0}{\beta_2^0}\right)$$ is well defined and, by the above remark on the degrees, $C$ is the critical point of $\cP^0$. Now, since $\aao\in\Omega_{\mathrm{gcd}(N_1,N_2)}$, by Corollary \ref{cor-lemma-gcd}, one has that $\deg_t(\gcd( {N}_1^0,{N}_2^0))=\deg_t(\gcd(N_1,N_2))>0$; recall that $\cP$ is normal. Now, the result follows from Theorem~\ref{thrm:genus0general} and~\cite[Theorem 6.22]{myBook}.
\end{proof}

Let us illustrate these ideas in an example.

\begin{example}\label{ex-1}
Let us consider $\K=\mathbb{Q}$ and $\FF=\LL:=\mathbb{Q}(a_1,a_2)$. Let

\vspace*{1mm}

\noindent
{\Small
\begin{align*}
F(\aa,x,y) =\, &((a_1^5 + a_2^5 + 3a_1^2a_2^2 - a_1a_2)y^2 + (2a_1^3a_2 + (-9a_2 - 1)a_1^2 + 3a_1 - 6a_2^4 + a_2^3)y + a_2^2(a_1 + 9a_2 - 3))x^3 \\ &+ ((-3a_1^3a_2^2 - 6a_1^4 + 3a_1^2a_2 - 6a_1a_2^2)y^2 + 9((a_2 + 2/9)a_1^2 + (-(8a_2)/9 - 1)a_1 - a_2^3/9 + 2a_2 + 2/9)a_1y \\ &+ 3(a_1 - 2/3)a_2^2)x^2 - 3(((a_2 - 4)a_1 - 2a_2^2)a_1y + a_1^3/3 - 3a_1^2 + (6a_2 + 4/3)a_1 - (8a_2)/3)a_1xy \\ &+ ((a_1^2a_2 - 8)y - 3a_1^2 + 2a_1)a_1^2y
\end{align*}
}


\noindent $\Cu(F)$ is a rational quintic that can be properly parametrized as
 \begin{equation}\label{eq-exP} \Pa(\aa,t)= \left(\dfrac{t a_{1}+2}{t^{2} a_{2}+t +a_{1}},
\dfrac{t +3}{t^{3} a_{1}+a_{2}}\right).
\end{equation}
The field of parametrization is $\LL$. We determine the open subset $\Omega_{\mathrm{proper}(\Pa)}$ (see Def. \eqref{def-omega}). Let us deal with $\Omega_1$. Clearly $\Omega_{\mathrm{def}(F)}=\mathbb{C}^2$. The homogeneous component of $F$ of maximum degree is
\[ \left(a_{1}^{5}+a_{2}^{5}+3 a_{1}^{2} a_{2}^{2}-a_{1} a_{2}\right) x^{3} y^{2}  \]
So,
$\Omega_{1}:= \mathbb{C}^2 \setminus  \V(a_{1}^{5}+a_{2}^{5}+3 a_{1}^{2} a_{2}^{2}-a_{1} a_{2}). $
One has that
$\Omega_2:=\mathbb{C}^2\setminus   \{(0,0)\}$. 
Note that, if $\aao\not\in \Omega_2$, the second component of $\Pa$ is not well-defined.
Let us deal with $\Omega_3$. The polynomials $G_i$, $G_{i}^{*}$, $R$ are
\[ \begin{array}{lcl} G_1&=& h^{2} t a_{1} a_{2}-h \,t^{2} a_{1} a_{2}+2 h^{2} a_{2}-h a_{1}^{2}-2 t^{2} a_{2}+t a_{1}^{2}+2 h -2 t \\
G_2& =& h^{3} t a_{1}-h \,t^{3} a_{1}+3 h^{3} a_{1}-3 t^{3} a_{1}-h a_{2}+t a_{2}\\
G&= & h-t \\
G_{1}^{*}&=& \left(h t a_{1}+2 h +2 t \right) a_{2}-a_{1}^{2}+2 \\
G_{2}^{*}&=& \left(\left(t +3\right) h^{2}+\left(t^{2}+3 t \right) h +3 t^{2}\right) a_{1}-a_{2}
\\
R&= & 3 h^{4} a_{1}^{3} a_{2}^{2}-2 h^{4} a_{1}^{2} a_{2}^{2}+h^{3} a_{1}^{4} a_{2}+6 h^{3} a_{1}^{2} a_{2}^{2}+3 h^{2} a_{1}^{4} a_{2}-h^{2} a_{1}^{2} a_{2}^{3}-2 h^{3} a_{1}^{2} a_{2}-2 h^{2} a_{1}^{3} a_{2}\\
&& +h a_{1}^{5}-6 h^{2} a_{1}^{2} a_{2}+12 h^{2} a_{1} a_{2}^{2}-6 h a_{1}^{3} a_{2}-4 h a_{1} a_{2}^{3}+3 a_{1}^{5}+4 h^{2} a_{1} a_{2}-4 h a_{1}^{3}+12 h a_{1} a_{2}\\&&-12 a_{1}^{3}-4 a_{2}^{3}+4 h a_{1}+12 a_{1}
\end{array}
\]
Moreover, $A_1=-h a_{1} a_{2}-2 a_{2},A_2=-h a_{1}-3 a_{1},B=-1$ (see Def. \eqref{def-gcd}).
Hence, $\Omega_1=\mathbb{C}^2$. Moreover, $\Omega_{\mathrm{nonZ}(A_1)}=\mathbb{C}^2\setminus \V(a_2)$, $\Omega_{\mathrm{nonZ}(A_2)}=\mathbb{C}^2\setminus \V(a_1)$ and $\Omega_{\mathrm{nonZ}(B)}=\mathbb{C}^2$. 
On the other hand, $\Omega_{\mathrm{nonZ}(R)}$ can be expressed as
\[ \mathbb{C}^2 \setminus \{ (0,0), (\pm\sqrt{2},0)\}. \]
Therefore, 
$$\Omega_3=\Omega_{\mathrm{gcd}(G_1,G_2)}=\mathbb{C}^2 \cap \left( \mathbb{C}^2 \setminus  \V(a_1)\right)
\cap \left( \mathbb{C}^2 \setminus \V(a_2)\right) \cap \left( \mathbb{C}^2 \setminus \{ (0,0), (\pm\sqrt{2},0)\}\right)= \mathbb{C}^2 \setminus  \V(a_1 a_2).$$
Finally, we deal with $\Omega_4$. We have
\[ \begin{array}{lcl}
p_1=t a_{1}+2 & & p_2=t +3 \\
q_1= t^{2} a_{2}+t +a_{1} & & q_2=t^{3} a_{1}+a_{2} \\
\gcd(p_1,q_1)=1 & &\gcd(p_2,q_2) =1 \\
\res_t(p_1,q_1)= a_{1}^{3}-2 a_{1}+4 a_{2} && \res_t(p_2,q_2)=a_{2}-27 a_{1}
\end{array}
\]
Thus, $\Omega_4=\mathbb{C}^2 \setminus \V(a_1 a_2 (a_{1}^{3}-2 a_{1}+4 a_{2})(a_{2}-27 a_{1})).
$
Summarizing (see Fig. \ref{fig-1}, left)
\begin{equation}\label{eq-omegaEx1}  \Omega_{\mathrm{proper}(\Pa)}=\mathbb{C}^2\setminus \V(a_1 a_2   (a_{1}^{5}+a_{2}^{5}+3 a_{1}^{2} a_{2}^{2}-a_{1} a_{2})(a_{1}^{3}-2 a_{1}+4 a_{2})(a_{2}-27 a_{1})).\
\end{equation}
 \begin{center}
 \begin{figure}[h]
 \includegraphics[width=6cm]{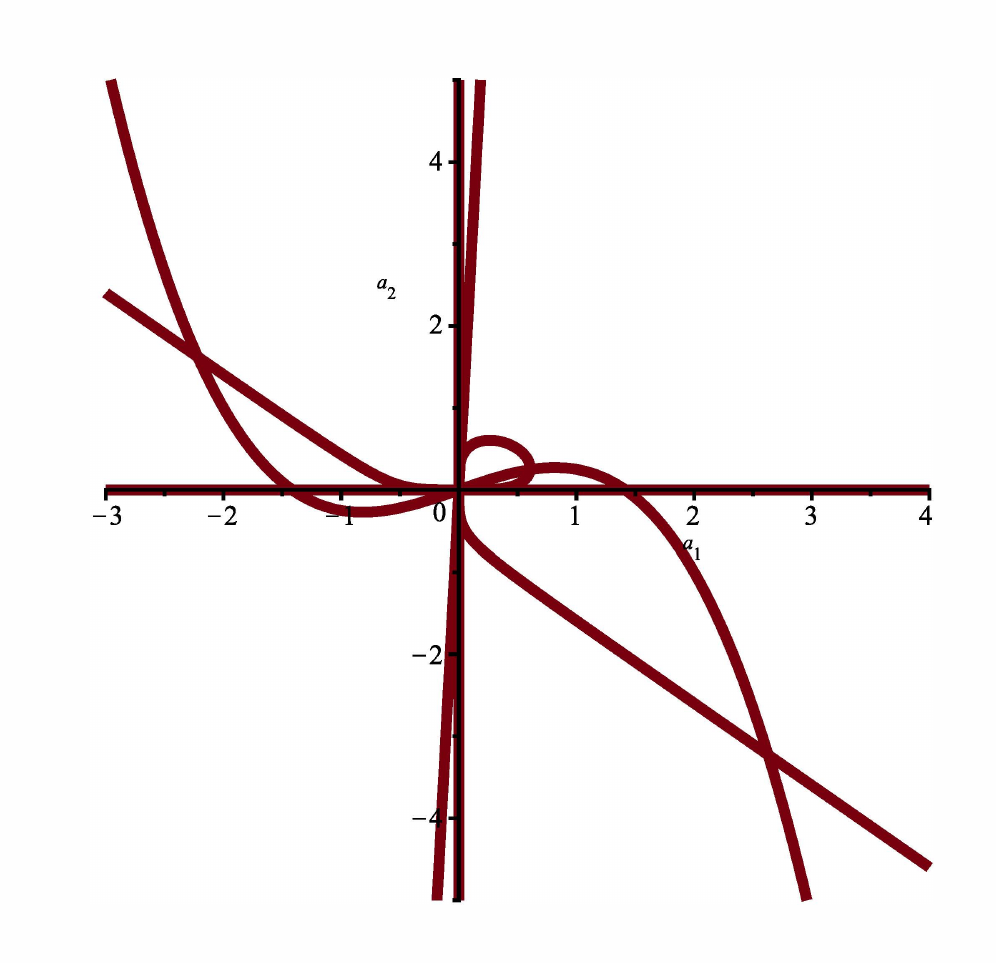}  \includegraphics[width=7cm]{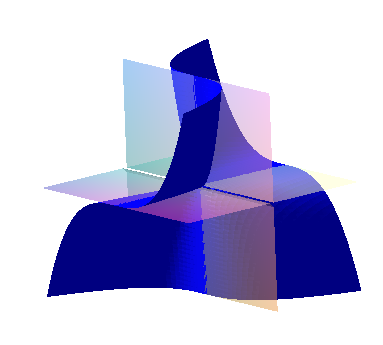} \caption{Left: Plot of the real part of the closed set defining $\Omega_{\mathrm{proper}(\Pa)}$ in Example \ref{ex-1}. Right:  Plot of the real part of the closed set defining $\Omega_{\mathrm{genus}(F)}$ in Example \ref{ex-cubic}}
 \label{fig-1}
 \end{figure}
 \end{center}
\end{example}

\section{Decomposition of $\S$}\label{sec:decomposition}

The goal in this section is to provide an algorithm decomposing the space $\S$ so that in each subset of the decomposition we can give information on the genus of the corresponding specialized curve.

Let  $F$ be as in~\eqref{eq-F}, irreducible over $\overline{\FF}$. We first compute the genus of $\Cu(F^h)$. Let $\gg:=\genus(\Cu(F^h))$. Furthermore, if $\gg=0$, let $\cP(\aa,\gamma,t)$ be, as in~\eqref{eq-Pa}, a proper parametrization of $\Cu(F^h)$. We consider the open subset
\begin{equation}\label{eq-Sigma1}
\Sigma:=\left\{ \begin{array}{cl}
 \Omega_{\mathrm{genus}(F)}  & \text{if $\gg>0$ (see Def. \eqref{def-genus})} \\
 \noalign{\vspace*{1mm}}
 \Omega_{\mathrm{proper}(\cP)} & \text{if $\gg=0$ (see Def. \eqref{def-omega})}
\end{array}\right.
\end{equation}
At this level of the process we know that (see Theorems~\eqref{thm:genusPreservation} and~\eqref{thrm:genus0general})
\begin{enumerate}
\item If $\gg>0$, then for $\aao\in \Sigma$ it holds that $\Cu(F,\aao)$ is either reducible or its genus is $\gg$.
\item If $\gg=0$, then for $\aao\in \Sigma$ it holds that $\Cu(F,\aao)$ is rational and $\cP(\aao,\gamma^0,t)$ parametrizes properly $\Cu(F,\aao)$.
\end{enumerate}
In the following, we analyze the specializations when working in the closed set
\begin{equation}\label{eq-Z}
\mathcal{Z}:=\S\setminus \Sigma.
\end{equation} First, let us discuss the computational issues that may appear.
Let $\mathcal{A}\subset \K[\aa]$ be a set of generators of $\mathcal{Z}$, and let $\mathrm{I}$ be the ideal generated by $\mathcal{A}$ in $\K[\aa]$.
We consider the prime decomposition of $\mathrm{I}$
$$\mathrm{I}=\bigcup_{j=1}^{\ell} \mathrm{I}_j.$$
Now, for each prime ideal $\mathrm{J}\in \{\mathrm{I}_1,\ldots,\mathrm{I}_\ell\}
$ we consider the quotient field of $\K[\aa]/\mathrm{J}$; we denote it by $\LL_{\mathrm{J}}$.  Elements in $\LL_{\mathrm{J}}$ are quotients of equivalence classes of $\overline{\K}[\aa]/\mathrm{J}$. We will assume that elements in $\overline{\K}[\aa]/\mathrm{J}$ are always expressed by means of a canonical representative of the class in the following sense. We fix a Gr\"obner basis   $\mathcal{G}$ of $\mathrm{J}$ w.r.t. some fixed order. Then, the elements in $\overline{\K}[\aa]/\mathrm{J}$ are uniquely represented by their normal form w.r.t. $\mathcal{G}$ (see e.g. Prop. 1 and Ex. 13, Chap. 2, Sect. 6. in \cite{cox}) and, hence, elements in $\LL_{\mathrm{J}}$ are represented as the quotient of the  canonical representatives of their numerators and denominators. So, by {abuse of notation}, we will identify, via the canonical representation, the elements in $\LL_{\mathrm{J}}$ with elements in $\overline{\K}(\aa)$. In addition, we consider an algebraic element $\gamma_\JJ$ over $\LL_{\JJ}$ and we denote by $\FF_\JJ$ the field $\FF_{\JJ}:=\LL_{\JJ}(\gamma_{\JJ})$.

We observe that $\FF_{\mathrm{J}}$ is a computable field with a polynomial factorization algorithm available; zero test and basic arithmetic (addition, multiplication and inverse computation) can be carried out e.g. by taking the normal forms w.r.t. a Gr\"obner basis of $\mathrm{J}$. For the polynomial factorization we refer to (see Section 10.2 and Appendix B in \cite{wang2}, see also \cite{wang1}). As a particular case, as in Example \ref{ex-1} and \ref{ex-2}, if $\V(\mathrm{J})$ is a rational variety, one may work over $\K(\mathcal{Q}(\lambda_1,\ldots,\lambda_m))$ instead of $\FF_{\mathrm{J}}$, where $\mathcal{Q}(\lambda_1,\ldots,\lambda_m)$ is a parametrization of $\V(\mathrm{J})$.

Concerning specializations, instead of working in $\S$ (see~\eqref{eq-S}), we take the parameter values in the irreducible variety $\V(\mathrm{J})$. Then, for $\aao\in \V(\mathrm{J})\subset \S$, and $f\in\overline{\K}[\aa]/\mathrm{J}$, we denote by $f(\aao)$ the specialization at $\aao$ of the equivalence class of $f$; note that since $\aao\in \V(\mathrm{J})$ the specialization does not depend on the representative. Similarly, if $f:=p/q\in\FF_{\mathrm{J}}$ and $q(\aao)\neq 0$, then $f(\aao):=p(\aao)/q(\aao)\in \overline{\K}$.

In this situation, for each prime ideal $\mathrm{J}\in \{\mathrm{I}_1,\ldots,\mathrm{I}_\ell\}$ we consider the polynomial $F$ in \eqref{eq-F} as a polynomial in $\FF_{\mathrm{J}}[x,y]$. To emphasize this fact, we write $F_\JJ$. First we check the irreducibility of $F_\JJ$ over the algebraic closure of $\FF_\JJ$. If $F_\JJ$ is reducible we can either stop the decomposition over this closed subset, and claim that the specialization over $\V(\JJ)$ is reducible, or continue the process with each irreducible factor of $F_\JJ$. For irreducible $F_\JJ$, the process continues, as in the initial step, by computing the genus of $\Cu(F_\JJ)$. Since in each iteration of the process the dimension of the variety $\V(\JJ)$ decreases, we, at the end, reach the zero--dimensional case, and the decomposition ends.

Let us say that a specialization \textit{degenerates} if either $F(\aao,\gamma^0,x,y)$ is not well--defined or $F(\aao,\gamma^0,x,y)\in \overline{\K}$.
As a result of the process described above, we find a disjoint decomposition
\begin{equation}\label{eq-decomp}
\S=\dot{\bigcup}_{i\in I} \S_i
\end{equation}
such that, for every specialization $\aao\in \S_i$, one of the following holds
\begin{enumerate}
\item the specialization degenerates;
\item the genus is positive and preserved, or the specialized curve is reducible;
\item the genus is zero and a proper parametrization of $\Cu(F,\aao)$ is provided.
\end{enumerate}

\begin{remark}\label{rem:decomposition}
Let us remark that in the decomposition~\eqref{eq-decomp} we can take the union of those $\S_i$ corresponding to each of the three items above; say $\S_1, \S_2, \S_3$ representing the corresponding item. 
In this way, we can achieve a unique decomposition of the parameter space $\S$.
The $\S_i$ obtained in this way are constructible sets of $\S$, and $\S_2,\S_3$ are a finite union of subsets $\Sigma$ as in~\eqref{eq-Sigma1} and $\S_1$ is a closed subset directly defined from the implicit equation $F$.

Moreover, $\S_3$ can further be decomposed into a finite union of subsets $\S_{3,j}$ such that for every $j$, there is a proper parametrization $\Pa_j$ which is well--defined for every $\aao \in \S_{3,j}$ and specialized properly.

Finally, since $\Sigma$ of $F$ as in~\eqref{eq-Sigma1} is open and non-empty, depending on the genus of $F$, either $\S_2$ or $\S_3$ is a dense subset of $\S$.
\end{remark}

We illustrate the previous ideas by continuing the analysis of Example \ref{ex-1}.

\begin{example}\label{ex-2} (Continuation of Example~\ref{ex-1})
Taking into account \eqref{eq-omegaEx1}, the closed set $\mathcal{Z}$ (see \eqref{eq-Z}) decomposes as
\[ \mathcal{Z}= \V(a_1) \cup \V(a_2) \cup \V(a_{2}-27 a_{1}) \cup \V(a_{1}^{3}-2 a_{1}+4 a_{2})\cup \V(a_{1}^{5}+a_{2}^{5}+3 a_{1}^{2} a_{2}^{2}-a_{1} a_{2})\subset \overline{\mathbb{Q}}^2.
\]
We start with $\mathrm{J}_1:= <\!\!a_1\!\!>$ and $\V_1:=\V(\JJ_1)$. Since $\V_1$ is rational, surjectively parametrized by $\Qa_1:=(0,\lambda)$, we work over the field $\mathbb{Q}(\lambda)[x,y]$. We have that
\[ F_{\mathrm{J}_1}:=\lambda^{2} x^{2} \left(\lambda^{3} x \,y^{2}-6 \lambda^{2} x y +\lambda x y +9 \lambda x -3 x -2\right)
\]
and therefore all specializations in $\V(a_1)$ lead to a reducible curve. Additionally, one may distinguish the cases $\lambda=0$, that corresponds to the point $(0,0)$, where the specialization degenerates, and $\lambda\neq 0$ where $\Cu(F,\aao)$ decomposes to the union of a double line and a rational cubic.

The analysis for $\mathrm{J}_2:=<\!\!a_2\!\!>$, and $\V_2:=\V(\JJ_2)$ looks similar. Since $\V_2$ is rational, parametrized by $\Qa_2:=(\lambda,0)$, we work over the field
$\mathbb{Q}(\lambda)[x,y]$. We have that

\vspace*{1mm}

\noindent
{\small
\begin{align*}
F_{\mathrm{J}_2} := \, &\lambda y (\lambda^{4} x^{3} y -6 \lambda^{3} x^{2} y -\lambda^{3} x +2 \lambda^{2} x^{2}+12 \lambda^{2} x y -\lambda \,x^{3}-3 \lambda^{3}+9 \lambda^{2} x -9 \lambda \,x^{2}+3 x^{3}+2 \lambda^{2}-4 \lambda x -8 \lambda y +2 x^{2}).
\end{align*}
}


\noindent Thus, all specializations in $\V_2$ lead to a reducible curve; note that $\mathcal{Q}_2$ is surjective. The case $\lambda=0$ is covered above, and for $\lambda\neq 0$, the specialization $\Cu(F,\aao)$ decomposes to the union of a line and a rational quartic.

Let us study $\mathrm{J}_3:=<\!\! a_{2}-27 a_{1} \!\!>$ and $\V_3:=\V(\JJ_3)$. Again, $\V_3$ is rational, parametrized by $\Qa_3:=(\lambda,27\lambda)$, and we work over the field
$\mathbb{Q}(\lambda)[x,y]$. We have that

\vspace*{1mm}

\noindent
{\small
\begin{align*}
F_{\mathrm{J}_3} := \, &\lambda (244 \lambda  x +3 \lambda -3 x -2 ) (58807 \lambda^{3} x^{2} y^{2}-732 \lambda^{3} x \,y^{2}+732 \lambda^{2} x^{2} y^{2}+9 \lambda^{3} y^{2}-13068 \lambda^{2} x^{2} y \\ &+464 \lambda^{2} x \,y^{2}+9 \lambda  \,x^{2} y^{2}+81 \lambda^{2} x y +6 \lambda^{2} y^{2}-81 \lambda  \,x^{2} y -6 \lambda x \,y^{2}-\lambda^{2} y +729 \lambda  \,x^{2}-106 \lambda  x y +4 \lambda  \,y^{2}-x^{2} y ).
\end{align*}
}


\noindent The analysis of $\V_3$ is identical to $\V_2$.

Let us study $\mathrm{J}_4:=<\!\! a_{1}^{3}-2 a_{1}+4 a_{2} \!\!>$ and $\V_4:=\V(\mathrm{J_4})$ that is again rational and it is properly and surjectively parametrized by $\Qa_4:=(\lambda, -\frac{1}{4} \lambda^{3}+\frac{1}{2} \lambda)$. We work over the field $\mathbb{Q}(\lambda)[x,y]$. We have that

\vspace*{1mm}

\noindent
{\small
\begin{align*}
F_{\mathrm{J}_4} := \, & \lambda  (\lambda^{5} y -2 \lambda^{3} y +12 \lambda^{2}-8 \lambda +32 y) (\lambda^{9} x^{3} y -8 \lambda^{7} x^{3} y +12 \lambda^{6} x^{3}+24 \lambda^{5} x^{3} y +8 \lambda^{5} x^{3}-32 \lambda^{4} x^{3} y \\ &-72 \lambda^{4} x^{3}-32 \lambda^{3} x^{3} y -16 \lambda^{4} x^{2}-32 \lambda^{3} x^{3}+192 \lambda^{3} x^{2} y +144 \lambda^{2} x^{3}+16 \lambda  \,x^{3} y +64 \lambda^{2} x^{2}-384 \lambda^{2} x y \\ &+32 \lambda  \,x^{3}-96 x^{3}+256 \lambda  y -64 x^{2} ).
\end{align*}
}


\noindent Again $\V_4$ behaves as $\V_2$ and $\V_3$.
Finally, let us analyze $\mathrm{J}_5:=<\!\! a_{1}^{5}+a_{2}^{5}+3 a_{1}^{2} a_{2}^{2}-a_{1} a_{2}\!\!>$ and $\V_5:=\V(\JJ_5)$ which is a rational quintic and properly and surjectively parametrized as
{\small
\[ \Qa_5(\lambda):=\left({\frac { \left( 5\,\lambda-1 \right)  \left( 5\,\lambda-2 \right) ^{4
}}{3125\,{\lambda}^{4}-3750\,{\lambda}^{3}+1750\,{\lambda}^{2}-375\,
\lambda+31}},-{\frac { \left( 5\,\lambda-2 \right)  \left( 5\,\lambda-
1 \right) ^{4}}{3125\,{\lambda}^{4}-3750\,{\lambda}^{3}+1750\,{\lambda
}^{2}-375\,\lambda+31}}
\right). \]
}
Those values for which the parametrization is not defined, i.e. the poles, play no role in this analysis.
The polynomial $F_{\JJ_5}$ is $F_{\JJ_5}(\lambda,x,y)=F(\Qa_5(\lambda),x,y)\in \mathbb{Q}(\lambda)[x,y]$.
It holds that $\deg(\Cu(F_{\JJ_5}))=4$, $\genus(\Cu(F_{\JJ_5}))=0$ and a proper surjective parametrization is $\cP_{\JJ_5}(\lambda,t):=\cP(\Qa(\lambda),t)$ (see \eqref{eq-exP}). So, we get (see Def. \ref{def-omega})
\[ \Omega_{\mathrm{proper}(\cP_{\JJ_5})}:=\V_5 \setminus \{ \cP_{\JJ_5}(\lambda_0)\,|\, f(\lambda_0)=0\} \]
where $f:=p_1\, p_2\,p_3\, p_4$ and
\[ \begin{array}{ll}p_1:=&  5\,\lambda-1,\,\,\,
p_2:=   5\,
\lambda-2,\,\,\,
p_3:=   25\,{\lambda}^{2}-15\,\lambda+3,
\\
p_4:= & \text{\small{$
 390625\,{\lambda}^{8}-937500\,{\lambda}^{7}+1343750\,{\lambda}
^{6}-1237500\,{\lambda}^{5}+711875\,{\lambda}^{4}-253500\,{\lambda}^{3
}$}}\\ &\text{\small{$+54475\,{\lambda}^{2}-6495\,\lambda+331.$}} \end{array}
 \]
By Theorem \ref{thrm:genus0general}, for every $\aao\in \Omega_{\mathrm{proper}(\cP_{\JJ_5})}$ it holds that $\Cu(F,\aao)$ is rationally parametrized by $\cP(\aao,t)$ (see \eqref{eq-exP}) or, equivalently, by $\cP_{\JJ_5}(\Qa^{-1}(\aao),t)$. Now, let us analyze the curve in $\mathcal{Z}_5:= \V_5\setminus \Omega_{\mathrm{proper}(\cP_{\JJ_5})}=\{ \cP_{\JJ_5}(\lambda_0)\,|\, f(\lambda_0)=0\}$ (see \eqref{eq-Z}). 
We define $\aao_1:=(0,0)=\cP_{\JJ_5}(\lambda_0)$, where $\lambda_{0}$ is a root of $p_1 \,p_2$; $\aao_2:=(-1,-1)=\cP_{\JJ_5}(\lambda_0)$, where $\lambda_{0}$ is a root of $p_3$; and observe that $p_4$ generates 8 points on the curve that we denote by $\aao_{i}, i\in \{3,\ldots,10\}$ and which correspond to $\cP_{\JJ_5}(\lambda_0)$ where $\lambda_{0}$ is one of the roots of $p_4$. Thus, $\mathcal{Z}=\{\aao_1,\ldots,\aao_{10}\}$, $\Cu(F,\aao_1)=\mathbb{C}^2$, and, for $i\in \{2,\ldots,10\}$, $\Cu(F,\aao_i)$ are rational cubics parametrized by $\cP_{\JJ_5}(\aao_i,t)$.

Summarizing, $\S$ decomposes as (see \eqref{eq-decomp} and Remark~\ref{rem:decomposition})
\begin{align*}
\S = &\left(\S_1:=\{(0,0)\}\right) \cup \left(\S_2:=\cup_{i=1}^{4}\V_i \setminus \{(0,0)\} \right) \\ & \cup \left(\S_{3,1}:=\Omega_{\mathrm{proper}(\cP)}\right) \cup \left(\S_{3,2}:=\Omega_{\mathrm{proper}(\cP_{\JJ_5})}\right) \cup \left(\S_{3,3}:=\{\aao_2,\ldots,\aao_{10}\}\right).
\end{align*}
For $\aao \in \S_1$, $\Cu(F,\aao)$ degenerates. 
For $\aao \in \S_2$, $\Cu(F,\aao)$ is reducible (note that $\aao_1\in \S_2$). 
For $\aao\in \S_{3,1}$, the specialized curve $\Cu(F,\aao)$ is a quintic parametrized by $\Pa(\aao,t)$; for $\aao \in \S_{3,2}$, $\Cu(F,\aao)$ is a quartic parametrized by $\Pa_{\mathrm{J}_5}(\aao,t)$; and for $\aao \in \S_{3,3}$, $\Cu(F,\aao)$ is a cubic parametrized by $\Pa_{\mathrm{J}_5}(\aao,t)$.
\end{example}

\begin{example}\label{ex-cubic}
Let $\K=\Q$ and $\FF=\LL=\Q(a_1,a_2,a_3)$. We consider
\[ F={x}^{3}+{x}^{2}a_{{1}}+{y}^{3}+a_{{2}}a_{{3}}\in \FF[x,y]. \]
One has that $\genus(\Cu(F))=1$. Using the ideas in Section~\ref{sec-genus}, we compute $\Omega_{\mathrm{genus}(F)}$ (see Def.~\ref{def-genus}). We observe that since $\Cu(F)$ is an elliptic cubic, no blowup is required and, hence, $\Omega_{\mathrm{genus}(F)}=
\Omega_{\mathrm{genusOrd}(F)}$. Indeed, one gets that (see Fig. \ref{fig-1}, right)
\[ \Omega_{\mathrm{genus}(F)}:=\C^3 \setminus \V\left(-a_{{2}}a_{{3}} \left( -4\,{a_{{1}}}^{3}-27\,a_{{2}}a_{{3}} \right)\right).\]
Thus, by Theorem \ref{theorem:genus-strong}, for every $\aao\in \Omega_{\mathrm{genus}(F)}$, $\Cu(F,\aao)$ is either reducible or it is a genus 1 cubic curve. Let us analyze the specializations in $\mathcal{Z}:=\C^3\setminus  \Omega_{\mathrm{genus}(F)}=\V\left(-a_{{2}}a_{{3}} \left( -4\,{a_{{1}}}^{3}-27\,a_{{2}}a_{{3}} \right)\right)$.

Let $\JJ_1:=<\!\!a_2\!\!>$ and $\V_1:=\V(\JJ_1)$. Then, $F_{\JJ_1}:={x}^{3}+a_{{1}}{x}^{2}+{y}^{3}$. We observe that $\genus(\Cu(F_{\JJ_1}))=0$ and
\[\cP_{\JJ_1}:=\left(-{\frac {{t}^{3}a_{{1}}}{{t}^{3}+1}},-{\frac {{t}^{2}a_{{1}}}{{t}^{3}
+1}}
\right) \]
is a proper parametrization of $\Cu(F_{\JJ_1})$. Applying Def. \ref{def-omega} to $F_{\JJ_1}$ and $\cP_{\JJ_1}$ we get that $\Omega_{\mathrm{proper}(\cP_{\JJ_1})}:=\V_1\setminus \V(a_1)$. Therefore, by Theorem \ref{thrm:genus0general}, for all $\aao\in \Omega_{\mathrm{proper}(\cP_{\JJ_1})}$ it holds that $\Cu(F,\aao)$ is a rational curve parametrized by $\cP_{\JJ_1}(\aao,t)$. However, for the remaining case, namely the points $(0,0,\mu)$ for $\mu\in\C$,  $\Cu(F,(0,0,\mu))$ decomposes as the product of three lines.

Let $\JJ_2:=<\!\!a_3\!\!>$ and $\V_2:=\V(\JJ_2)$. Now, the situation is identical to the previous case. Let $\JJ_3:=<\!\!-4\,{a_{{1}}}^{3}-27\,a_{{2}}a_{{3}}\!\!>$ and $\V_3:=\V(\JJ_3)$. The surface $\V_3$ can be properly parametrized as
\[ \Qa(\lambda_1,\lambda_2)=\left(\lambda_{{1}},\lambda_{{2}},-{\frac {4\,{\lambda_{{1}}}^{3}}{27\,\lambda_{{2}}}}\right). \]
We observe that $\Qa$ is not surjective. Indeed, $\Qa(\C^2)=\V_3\setminus \{(0,0,\mu)\,|\, \mu\in \C\}$ (see e.g. Remark 3 in \cite{SSV}). 
However, the specializations in $\{(0,0,\mu)\,|\, \mu\in \C\}$ have already been analyzed. 
So, we treat the case $\Qa(\C^2)$. Then, $F_{\JJ_3}$ can be taken as 
$$F_{\JJ_3}:=F(\Qa(\lambda_1,\lambda_2),t)= {x}^{3}+{x}^{2}\lambda_{{1}}+{y}^{3}-{\frac {4\,{\lambda_{{1}}}^{3}}{27}},$$ where $(\lambda_1,\lambda_2)\in \C^2 \setminus \{(0,0)\}$. It holds that $\genus(\Cu(F_{\JJ_3}))=0$ and 
\[ \cP_{\JJ_3}:=\left({\frac {{t}^{3}\lambda_{{1}}-2\,\lambda_{{1}}}{3\,{t}^{3}+3}},{\frac
{{t}^{2}\lambda_{{1}}}{{t}^{3}+1}}
\right) \]
is a proper parametrization. 
Applying Def. \ref{def-omega} to $F_{\JJ_3}$ and $\cP_{\JJ_3}$, we get that $\Omega_{\mathrm{proper}(\cP_{\JJ_3})}=\V_3\setminus \{(0,0,\mu)\,|\, \mu\in \C\}$. Therefore, by Theorem \ref{thrm:genus0general}, for all $\aao\in \Omega_{\mathrm{proper}(\cP_{\JJ_3})}$ it holds that $\Cu(F,\aao)$ is a rational curve parametrized by $\cP_{\JJ_3}(\aao,t)$.

Summarizing, $\S$ decomposes as (see \eqref{eq-decomp})
\[ \S=\left(\S_2:=\Omega_{\mathrm{genus}(F)} \cup \{(0,0,\mu)\,|\, \mu\in \C\}\right) \cup \left(\S_{3}:=\Omega_{\mathrm{proper}(\cP_{\JJ_1})}\right). \]
For $\aao\in \S_2, \Cu(F,\aao)$ is either reducible or an elliptic curve; and for $\aao\in \S_3, \Cu(F,\aao)$ is a rational cubic parametrized by $\cP_{\JJ_3}(\aao,t)$.
\end{example}

\section{Some Illustrating Applications}\label{sec-application}
In this section, we illustrate by examples some possible applications of the theory developed in the paper.
In the first example, given a surface, we consider the problem of determining its rational level curves, if any.

\begin{example}\label{ex-app1} \textsf{Level curves of a surface.} \\
Let $\mathcal{S}$ be the surface defined over $\C$ by the polynomial
\[ F=x^{6}-5 x^{4} y +3 x^{4} z -y^{5}+2 y^{4} z -y^{3} z^{2}-x^{3} z +5 x^{2} y^{2}-7 x^{2} y z +3 x^{2} z^{2}+y^{2} z -2 y z^{2}+z^{3}-x^{2} \in \Q(z)[x,y]. \]
So, with the terminology of the paper, $\aa=z$, $\K=\Q$ and $\LL=\FF=\Q(z)$. With this interpretation, the idea is to analyze the genus of $\Cu(F)$ under specializations in $\S:=\C$. For this purpose, we first compute $\Stand(F^h)$ as in~\eqref{eq-SingularLocusF}. Following the steps  in Subsection~\ref{subsec-genus}, we get
\[ \Stand(F^h)=\left\{\left(0:z: 1\right) \right\}_{m_1=t}. \] 
The family consists in one ordinary double point. Therefore,  $\genus(\Cu(F))=9$. Since the singularities are all ordinary, we get that (see Def. \ref{def-genus}) $\Omega_{\mathrm{genus}(F)}=\Omega_{\mathrm{genusOrd}(F)}.$
Moreover (see Def. \ref{def:genusOrd}),
\[ \begin{array}{ll} \Omega_{\mathrm{genusOrd}(F)}= &\C\setminus \left\{ -1, 0, 1\right\}. \end{array}\]
Using Theorem \ref{thm:genusPreservation}, for $z_0\in \Omega_{\mathrm{genusOrd}(F)}$, $\Cu(F,(x,y,z_0))$ is either reducible or its genus is 9. In any case, no rational level curve appears. For the elements in $\mathcal{Z}:=\C \setminus \Omega_{\mathrm{genusOrd}(F)}$, we get that $\Cu(F,(x,y,\pm 1))$ are irreducible of genus 7 and $\Cu(F,(x,y,0))$ is irreducible of genus 0. Indeed, $\Cu(F,(x,y,0))$ can be parametrized by $(t^{5}, t^{6}-t^{2})$.
\end{example}

In the second example, we consider the linear homotopy deformation of two curves and we analyze the genus of each instance curve.

\begin{example}\label{ex-app2} \textsf{Linear homotopy deformation of curves.} \\
Let us consider the linear homotopy between the Fermat cubic curve and the unit circle. That is, we consider the polynomial
\[ F(\lambda,x,y)=\left(1-\lambda \right) \left(x^{3}+y^{3}-1\right)+\lambda  \left(x^{2}+y^{2}-1\right)\in \Q(\lambda)[x,y]. \]

With the terminology of the paper, $\aa=\lambda$, $\K=\Q$ and $\LL=\FF=\Q(\lambda)$. Now, we  study the genus of $\Cu(F)$ under specializations in $\S:=\C$; or, in particular, in the real interval $[0,1]$. For this purpose, we first observe that $\genus(\Cu(F))=1$ and hence (see Def. \ref{def-genus}), $\Omega_{\mathrm{genus}(F)}=\Omega_{\mathrm{genusOrd}(F)}.$
We get that $\Omega_{\mathrm{genusOrd}(F)}=\C\setminus \V(g)$  
where 

\noindent {\small
\begin{align*}
g(\lambda) = &-\left(2 \lambda^{3}-5 \lambda^{2}+7 \lambda -3\right) \left(\lambda^{6}-4 \lambda^{5}+15 \lambda^{4}-29 \lambda^{3}+43 \lambda^{2}-33 \lambda +9\right) \left(4 \lambda -3\right) \left(-1+\lambda \right) \left(\lambda -3\right) \\ & \left(2 \lambda^{9}+27 \lambda^{8}+5049 \lambda^{7}-40068 \lambda^{6}+148716 \lambda^{5}-315657 \lambda^{4}+398763 \lambda^{3}-295245 \lambda^{2}+118098 \lambda -19683\right) \\ &\left(8 \lambda^{3}-27 \lambda^{2}+54 \lambda -27\right).
\end{align*}
} 


Using Theorem \ref{thm:genusPreservation}, for $\lambda_0\in \Omega_{\mathrm{genusOrd}(F)}$, $\Cu(F,(\lambda_0,x,y,\lambda))$ is either reducible or its genus is 1. In any case, no rational deformation instance appears. For the elements in $\mathcal{Z}=\C\setminus \Omega_{\mathrm{genusOrd}(F)}$, we get that $\Cu(F,(1,x,y))$ is rational, the specialized cubics $\Cu(F,(3/4,x,y))$ and $\Cu(F,(3,x,y))$ factor as a union of a line and a conic, and for all the other cases the genus remains one (see Fig~\ref{fig-2}).
\begin{center}
	\begin{figure}[h]
		\includegraphics[width=5cm]{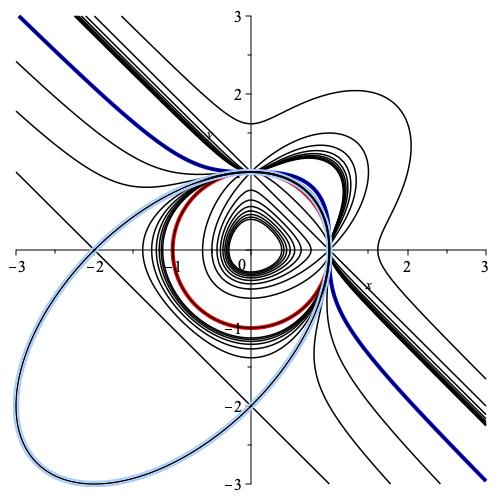}  \includegraphics[width=5cm]{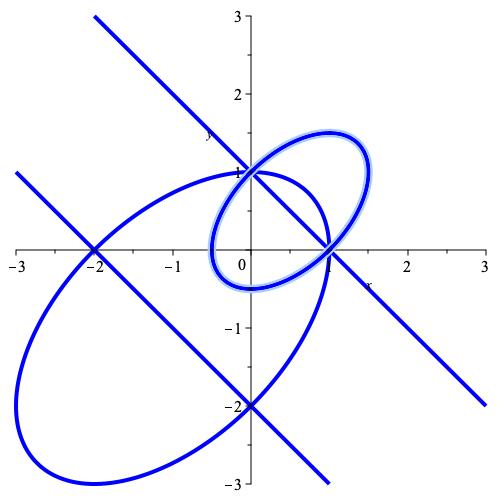} \caption{Left: Plot of the real part of different instances of the deformation in Example~\ref{ex-app2}. Right: Plot of the real part of $\Cu(F(3/4,x,y))$ and $\Cu(F,(3,x,y))$ in Example~\ref{ex-app2}.}
		\label{fig-2}
	\end{figure}
\end{center}
\end{example}

Let $\mathcal{O}$ be a connected open subset of $\C$ and let $\mathrm{Mer}(\mathcal{O})$ be the field of meromorphic functions in $\mathcal{O}$ (see \cite{meromorfas}). 
We consider a polynomial equation of the form
\begin{equation}\label{eq-fun}
\sum_{i,j\in {I}} f_{i,j}(t) x^i y^j=0
\end{equation}
where $I$ is a finite subset of $\mathbb{N}^2$, and where $f_{i,j}\in \mathrm{Mer}(\mathcal{O})$. Let $\mathbf{f}$ be the tuple with all the functions $f_{i,j}$ appearing in \eqref{eq-fun}. The question now is to decide, and indeed compute, whether there exists rational solutions of the equation; that is $p,q\in \overline{\C(\mathbf{f})}$ such that $\sum_{i,j\in {I}} f_{i,j}(t) p^i q^j=0$.

We may proceed as follows. We consider the polynomial $F(\aa,x,y)$ resulting from the formal replacement in \eqref{eq-fun} of each function $f_{i,j}$ by a parameter $a_{k}$. Now, with the terminology of the paper, we take $\K=\C(\mathbf{f})$ and $\LL=\FF=\K(\aa)$.

Then, we decompose $\S$ (see \eqref{eq-S}) as described in Section \ref{sec:decomposition}; note that the computations can be carried out over $\C(\aa)$ instead of over $\FF$. Then, if any subset in the decomposition has genus zero, and the functions $\mathbf{f}$ belong to it, we obtain a (family of) rational solutions. Let us see a particular example.
\begin{example}\label{ex-app3} \textsf{Rational solution of functional algebraic equations.} \\
We consider the functional algebraic equation
\begin{equation}\label{eq-funEx}
  x^{2} y^{3} f_{1}^{4}-x^{2} y^{2} f_{2}^{4} f_{3}-2 x \,y^{3} f_{1}^{2} f_{2}+2 x y f_{1} f_{2}^{2} f_{3}^{2}+y^{3} f_{2}^{2}-f_{1}^{2} f_{3}^{3}
=0, \end{equation}
where $\mathbf{f}=(f_1,f_2,f_3):=(\sin(t),\cos(t),\mathrm{e}^{t})$. We associate to \eqref{eq-funEx} the curve $\Cu(F)$ where
\[ F=x^{2} y^{3} a_{1}^{4}-x^{2} y^{2} a_{2}^{4} a_{3}-2 x \,y^{3} a_{1}^{2} a_{2}+2 x y a_{1} a_{2}^{2} a_{3}^{2}+y^{3} a_{2}^{2}-a_{1}^{2} a_{3}^{3}. \]
It holds that $\genus(\Cu(F))=0$ and that a proper parametrization is 
\begin{equation}\label{eq-P-ex}
 \cP(\aa,W)=\left(\frac{W^{3} a_{1}+a_{2}}{W a_{2}^{2}+a_{1}^{2}}, 
\frac{a_{3}}{W^{2}}\right). 
\end{equation}
The open subset in Def. \ref{def-omega} turns to be
\[ \Omega_{\mathrm{proper}(\cP)}= \overline{\C(\mathbf{f})}^3\setminus \V_{\overline{\C(\mathbf{f})}}\left(a_{1} a_{2}  \left(a_{1}-a_{2}\right) \left(a_{1}^{6}+a_{1}^{5} a_{2}+a_{1}^{4} a_{2}^{2}+a_{1}^{3} a_{2}^{3}+a_{1}^{2} a_{2}^{4}+a_{1} a_{2}^{5}+a_{2}^{6}\right) a_{3} \right). \]
Since $\mathbf{f}\in \Omega_{\mathrm{proper}(\cP)}$, by Theorem \ref{thrm:genus0general}, we have that
\begin{equation}\label{eq-solApp} \left\{ x=\frac{W^{3} \sin \! \left(t \right)+\cos \! \left(t \right)}{W \left(\cos^{2}\left(t \right)\right)+\sin^{2}\left(t \right)}
, y=\frac{{\mathrm e}^{t}}{W^{2}}
\right\},
\end{equation}
for every $W\in \overline{\C(\mathbf{f})}$ such that~\eqref{eq-solApp} is well--defined, is a rational solution of~\eqref{eq-funEx}. In fact, the last statement holds more generally for every $W\in \overline{\mathrm{Mer}(\C)}$. 
\end{example}

\appendix

\section{Geometric Genus and Rational Parametrizations}\label{App}

In this appendix we recall the main steps to compute the genus of an irreducible plane curve and, in the affirmative case, how to compute a proper parametrization. There exist different methods to that goal: the adjoint curve based method (see e.g.~\cite{walker}, \cite{mauro} and~\cite{myBook}), the method based on the anticanonical divisor (see~\cite{hoeij}) or the method based on Puiseux expansions (see~\cite{PW}), among others. In this paper we will follow the adjoint curve based method (see~\cite{walker}); more precisely, we will follow the symbolic approaches in Chapters 3, 4, 5 in~\cite{myBook}. A fundamental tool for that approach is the manipulation of conjugate families of points. This is why we include a subsection on the topic. A similar treatment of the problem can be performed using the other algorithmic approaches.  Throughout this appendix, let $G\in K[x,y]\setminus K$ be irreducible over $\overline{K}$; recall that $K$ is a field extension of $\K$ (see notation at the introduction of the paper).

\subsection{Conjugate families of points}\label{subsec-families}
The key tool for our symbolic computation of the genus is the notion of conjugate family of points (see Definition 3.15 in~\cite{myBook}). We adapt the definition to our purposes. We present the definitions taking $K$, the ground field of $\Cu(G)$, as the reference field. Nevertheless, one may take any field extension of $K$ as field of reference.

\begin{definition}\label{def-family}
The set $\mathcal{F}=\{ (f_1(\alpha):f_2(\alpha): f_3(\alpha))\,|\, m(\alpha)=0 \}\subset \mathbb{P}^2(\overline{K})$ is called a
\textit{$K$--conjugate family of points}, or for short \textit{$K$-family of points}, if
\begin{enumerate}
\item $f_1,f_2,f_3,m\in K[t]$ and $gcd(f_1,f_2,f_3,m)=1$,
\item at least one of the polynomials $f_i$ is not zero,
\item  $\deg_t(m)>0$ and $m$ is squarefree.
\end{enumerate}
The family is represented as $$\mathcal{F}:=\{ (f_1:f_2: f_3)\}_{m}$$ and $m(t)$ is called the \textit{defining polynomial of $\mathcal{F}$}.  If $m(t)$ is irreducible over $K$, we say that $\mathcal{F}$ is \textit{irreducible}. We say that $\mathcal{F}:=\{ (f_1:f_2: f_3)\}_{m}$ is in \textit{canonical form} if $f_1,f_2,f_3$ are reduced modulo $m(t)$. In this paper, we will always assume that families are irreducible and expressed in canonical form.
\end{definition}

\begin{remark}\label{rem-fam}\
\begin{enumerate}
\item A single point $(a:b:c)\in \mathbb{P}^2(K)$ can be seen as the  $K$-family   $\{(a:b:c)\}_{t}.$
\item Factoring the defining polynomial $m(t)$ over $K$, every family decomposes as a finite union of irreducible families.  
\item For $H\in K[x,y,z]$ and $\mathcal{F}:=\{ (f_1 :f_2 : f_3 )\}_{m )}$ irreducible, $H(f_1(t),f_2(t),f_3(t))\neq 0$ mod $m(t)$ iff $H$ does not vanish at any point in the family $\mathcal{F}$.
%
%
\end{enumerate}
\end{remark}

\begin{definition}\label{def-family-curve}
We say that the irreducible $K$--family $\mathcal{F}:=\{ (f_1:f_2: f_3)\}_{m}$ is a \textit{family of points of $\Cu(G^h)$} if $G^h(f_1(t),f_2(t),f_3(t))=0$ modulo $m(t)$.  
\end{definition}

\begin{definition}\label{def-curve-fam}
Let $\fam=\{(f_1:f_2:f_3)\}_{m}$ be an irreducible $K$-family of $\Cu(G^h)$. Let $K_m$ be the quotient field of the integral domain $K[t]/\!<\!m\!>$. We defined the \textit{curve associated to $\fam$}, and we denote it by  $\CuFam(G^{h})$, as the curve defined by $G^h$ over $\overline{K_m}$. Note that $(f_1:f_2:f_3)$ is a point on $\CuFam(G^h)$. By abuse of notation, we will say that $\mathcal{F}$ is a point of $\CuFam(G^{h})$ meaning $(f_1:f_2:f_3)\in \CuFam(G^{h})$.
\end{definition}

Taking into account Remark~\ref{rem-fam} (3), one deduces that all points of an irreducible family have the same multiplicity as points of the curve. This motivates the next concept.

\begin{definition}\label{def-family-sing}
Let $\mathcal{F}:=\{ (f_1:f_2: f_3)\}_{m}$ be an irreducible $K$--family of points of $\Cu(G^h)$. We say that $\fam$ is a \textit{family of $r$-fold points} of $\Cu(G^h)$ if all points of $\fam$ are $r$-fold points of $\Cu(G^h)$. We denote by $\mathrm{mult}(\mathcal{F})$ the multiplicity of the points in $\mathcal{F}$, as points of $\Cu(G^h)$.
\end{definition}

\begin{remark}\label{re-mult-fam}
By Remark~\ref{rem-fam} (3), the multiplicity of an irreducible family $\fam$ can be computed by checking which is the first partial derivative of $G^h$ not vanishing, modulo $m(t)$, at the family. Alternatively, one can determine the multiplicity of  $\fam$ at a point of the curve associated to $\fam$ (see Definition \ref{def-curve-fam}).
\end{remark}

 
\begin{definition}\label{def-family-ord}
Let $\mathcal{F}:=\{ (f_1:f_2: f_3)\}_{m}$ be an irreducible $K$--family of $r$-fold points of $\Cu(G^h)$. We say that $\fam$ is \textit{ordinary} if all points of  $\fam$ are ordinary singularities in $\Cu(G^h)$. We say that $\fam$ is  \textit{non-ordinary} it is not an ordinary family.
\end{definition}
 
\begin{remark}\label{re-mult-ord}
The character of the irreducible family $\mathcal{F}:=\{ (f_1:f_2: f_3)\}_{m}$ can be deduced analyzing the squarefreeness, modulo $m(t)$, of the polynomial defining the tangents to $\Cu(G^h)$ at $(f_1:f_2:f_3)$. Alternatively, one can work with the curve $\CuFam(G^{h})$ (see Def. \ref{def-curve-fam}). Now, the family $\mathcal{F}$ in $\Cu(G^h)$ turns to be one point in $\CuFam(G^{h})$. Then, the character and multiplicity of $\mathcal{F}$ is the character and multiplicity of $(f_1:f_2:f_3)$ as a point in $\CuFam(G^{h})$.
\end{remark}

\subsection{Standard Decomposition of the Singular Locus}\label{subsec-standard-decomp}
 
\begin{definition}\label{def-stand-decomp}
A decomposition of the singular locus of the irreducible curve $\Cu(G^h)$ 
as a finite union of irreducible families of $K$-points, such that in each family the multiplicity and the singularity character is preserved, is called \textit{a $K$--standard decomposition of the singular locus of $\Cu(G^h)$} (see Remark~\ref{rem-standard-decomp-final} below). 
\end{definition}

In  \cite{myBook}, Section 3.3,  it is proved that the singular locus of $\Cu(G^h)$ always admits a $K$--standard decomposition. In the following, we describe a process to determine such a decomposution; see also \cite{myBook} page 83. First, we introduce the notion of regular position. 
\begin{definition}\label{def-regular}
We say that the affine plane curve $\Cu(G)$ is in \textit{regular position} w.r.t. $x$ if
\begin{enumerate}
\item the coefficient of $y^{\deg(G)}$ in $G$ is a non-zero constant; and
\item $G(a,b_i)=G_x(a,b_i)=0$ with $i\in \{1,2\}$ implies that $b_1=b_2$.
\end{enumerate}
\end{definition}

\begin{remark}\label{rem-regular-position}
\
\begin{enumerate}
\item
If $\Cu(G)$ is not in regular position, we may apply a linear change of coordinates over $\K$ such that $\Cu(G)$ is transformed into regular position (see e.g. Lemma 2 and Remark 3 in \cite{FS90}). 
\item For checking the regularity one may proceed as follows.  
Condition (1) in Def.  \eqref{def-regular} is easy to check. Let us now deal with the second condition.
Let  $\mathrm{J}$ be the ideal generated by $\{G,G_x\}$ in $K[x,y]$. Since $G$ is irreducible over $\overline{K}$,  $\mathrm{J}$ is zero--dimensional. Therefore, using the Shape lemma (see \cite{winkler} page  194), the normed reduced Gr\"obner basis of $\sqrt{\mathrm{J}}$, w.r.t. the lexicographic order $x<y$, is of the form
$ \{ u(x),y-v(x)  \}, $
with $u$ square-free and $\deg(v)<\deg(u)$, if and only if $\sqrt{\mathrm{J}}$, or equivalently $\mathrm{J}$, satisfies (2) in Def. \eqref{def-regular}. It remains to have a computational approach to determine $\sqrt{\mathrm{J}}$. This can be achieved, for instance, using Seidenberg lemma (see e.g.~\cite{seidenberg} or~\cite{laplagne}). 
More precisely, if $\mathrm{J}\cap K[x]=<\!f(x)\!>$ and $\mathrm{J}\cap K[y]=<\!g(y)\!>$, then
$\sqrt{\mathrm{J}}=<\!G, G_x, \tilde{f}, \tilde{g}\!>,$ 
where $\tilde{f}= f/\gcd(f,f')$ and $\tilde{g}= g/\gcd(g,g')$ and where $f'$ and $g'$ denotes the derivatives of $f$ and $g$ w.r.t. $x$ and $y$, respectively.
\end{enumerate}
\end{remark}

The process for computing a standard decomposition of the singular locus of $\Cu(G^h)$ is as follows.

 \vspace*{1mm}

\noindent \textsf{Step 1: Regular position.} If $\Cu(G)$ is not in regular position, apply an affine linear change of coordinates over $\K$, say $\mathcal{L}$, such that $\Cu(G)$ is transformed into regular position (see Rem. \ref{rem-regular-position} (1)).

\vspace*{1mm}

\noindent \textsf{Step 2: Families of singularities at infinity}. Factor $G^h(x,y,0)$ over $K$,
\[ G^h(x,y,0)=\prod_{i\in I} g_i(x,y)^{e_i}. \]
Note that none of the polynomials $g_i$ is $x$ because $\Cu(G)$ is in regular position w.r.t. $x$.
Then, the set of points at infinity decomposes in irreducible $K$--families as
\begin{equation*} \bigcup_{i\in I} \{(1:t:0)\}_{g_{i}(1,t)}.
\end{equation*}
Now, a family $\{(1:t:0)\}_{g_{i}(1,t)}$ is singular if and only if $G^{h}_{x}(1,t,0)=G^{h}_{y}(1,t,0)=G^{h}_{z}(1,t,0)=0$ modulo $g_{i}(1,t)$. Let $\mathcal{I}$ be the set containing all $g_{i}(1,t)$ defining singularities. Then, the irreducible families of $K$--singularities at infinity are
\begin{equation}\label{eq-sing-infinity} \bigcup_{m\in \mathcal{I}} \{(1:t:0)\}_{m(t)}.
\end{equation}

\vspace*{1mm}

\noindent \textsf{Step 3: Families of affine singularities.}
Let $\mathrm{I}$ be the ideal generated by $\{G, G_x, G_y\}$.
Since $\sqrt{\mathrm{I}}$ is regular w.r.t. $x$ because of condition (2) in Def. \ref{def-regular}, and zero--dimensional, by the discussion above on the Shape lemma, the normed reduced Gr\"obner basis w.r.t. the lexicographic order with $x<y$ of $\sqrt{\mathrm{I}}$ is of the form $\{A(x),y-B(x)\}$ with $A$ square-free and $\deg(B)<\deg(A)$; recall that if $\mathrm{I}\cap K[x]=<\!f(x)\!>$ and $\mathrm{I}\cap K[y]=<\!g(y)\!>$, then
\begin{equation*}\label{eq-radxy} \sqrt{\mathrm{I}}=<\!G, G_x, G_y, \tilde{f}, \tilde{g}\!>,
\end{equation*}
where $\tilde{f}, \tilde{g}, f', g'$ are as in Remark \ref{rem-regular-position} (2).
So, each irreducible factor $m(x)$ of $A(x)$ over $K$ generates the irreducible family
$\{(t:B(t):1)\}_{m(t)}.$ 
Therefore, if $\mathcal{A}$ denotes the set of all irreducible factors of $A$, the affine singularities decompose in irreducible $K$--families as
\begin{equation}\label{eq-sing-affine}
\bigcup_{a\in \mathcal{A}} \{(t:B(t):1)\}_{m(t)}.
\end{equation}

\vspace*{1mm}

\noindent \textsf{Step 4: Standard decomposition}. Applying $\mathcal{L}^{-1}$ to \eqref{eq-sing-infinity} and to \eqref{eq-sing-affine} we get a $K$--standard decomposition of the singular locus of $\Cu(G^h)$.

\begin{remark}\label{rem-standard-decomp-final}
Throughout the paper, we will use the $K$-standard decomposition achieved by means of the previous reasoning; we denote it as $\Stand(G^h))$. We observe that  
\begin{equation}\label{eq-SingularLocusApp}
\Stand(G^h):=\bigcup_{m(t)\in \mathcal{A}_a} \{(f_{1,m}:f_{2,m}:1)\}_{m} \, \cup \,  \bigcup_{m\in \mathcal{A}_{\infty}} \{(L_{1,m}:L_{2,m}:0)\}_{m}
\end{equation}
where $f_{i,m}(t),m(t)\in K[t]$, $L_{i,m}(t)\in \K[t]$  with $\deg(L_{i,m}(t))\leq 1$, $\gcd(L_{1,m}(t),L_{2,m}(t))=1$, and  $\mathcal{A}_a$ and $\mathcal{A}_{\infty}$ are finite sets of irreducible polynomials in $K[t]$. 
\end{remark} 
 
\subsection{Genus computation}\label{subsec-genus} We recall here the computation of the genus of an irreducible plane curve by means of the blow up of conjugate families of singular points (see Subsections \ref{subsec-families} and \ref{subsec-standard-decomp}). We denote the genus as either $\genus(\Cu(G))$ or $\genus(\Cu(G^h))$.  

Let $\Stand(G^h)$ be the $K$--standard decomposition of the singular locus of $\Cu(G^h)$ provided by the method in the previous subsection, see \eqref{eq-SingularLocusApp}. We  recursively, and separately, blow up each non--ordinary singularity via the neighboring singularities (see~\cite{walker}, \cite{mauro} or~\cite{myBook} for more details). Let us briefly recall how the first iteration step works, first for a single point, and afterwards for an irreducible family of conjugate points.   let $p\in \Cu(G^h)$ be a non--ordinary singular point.

\vspace*{1mm}

\noindent \underline{\textsf{Blow up of a point.}}  Let $p\in \Cu(G^h)$ be a non--ordinary singular point.

\vspace*{1mm}

\noindent \textsf{Step 1.} Apply a linear change of coordinates $\mathcal{L}$ such that: $p=\mathcal{L}(0:0:1)$, none of the tangent to $\Cu(G^h(\mathcal{L}))$ at $(0:0:1)$ is one of the lines $x=0,$ and $y=0$, and for $v\in \{x,y,z\}$ no point in $\left(\Cu(G^h(\mathcal{L}))\setminus \Cu(v)\right)\setminus \{(0:0:1)\}$ is a singularity of $\Cu(G^h)$.
\vspace*{1mm}

\noindent \textsf{Step 2.} Apply the Cremona transformation $\mathcal{Q}:=(yz:xz:xy)$ to $\Cu(G^h(\mathcal{L}))$. Then, $G^h(\mathcal{L}(\mathcal{Q}))$ factors as $G^h(\mathcal{Q}(\mathcal{L}))=x^{n_1}y^{n_2}z^{n_3} G^*$ for some natural numbers $n_1,n_2,n_3$. We call $G^*$ the \textit{quadratic transform of $G^h$ w.r.t. $p$}.

\vspace*{2mm}

The \textit{first neighboring} of $p$ is defined as
\[\left( \Cu(G^*) \cap \Cu(z)\right)\setminus \{(1:0:0),(0:1:0)\}.\]
The points, and their multiplicities, in the neighborhood, are in one to one correspondence with the tangents, and their multiplicities, to $\Cu(G^h)$ at $p$. The (first) neighboring singularities of $p$ are the neighboring points being singularities of $\Cu(G^*)$. The process continues till no non-ordinary neighboring point appears in the neighborhoods and until all non-ordinary singularities of $\Cu(G^h)$ have been blowed up. We will refer to the set of all singularities and neighboring singularities as the \textit{neighboring graph} of $\Cu(G^h)$.
In this situation, the genus can be computed
as (recall that $d=\deg(\Cu(G))$)
\begin{equation}\label{eq-genus}
\genus(\Cu(G))= \dfrac{(d-1)(d-2)}{2} -\dfrac{1}{2} \sum  \mathrm{mult}(p) (\mathrm{mult}(p)-1)
\end{equation}
where the sum is taking over all the points in the neighboring graph of $\Cu(G^h)$.

\vspace*{1mm}

\noindent \underline{\textsf{Blow up of a family.}}  
Let $\mathcal{F}=\{(f_1:f_2:f_3) \}_{m_0(t_0)}$ be a non-ordinary irreducible $K$--family of singularities. We may introduce the associated   curve $\CuFam(G^{h})$ (see Def. \ref{def-curve-fam}) and repeat the process with the single point $p_{\mathcal{F}}:=(f_1:f_2:f_3)$. In this way the first neighboring of $p_{\mathcal{F}}$ is decomposed into irreducible $K_{m_0}$--families; recall that $K_{m_0}$ is the quotient field of $K[t_0]/\!<\!m_0(t_0)\!>$. These irreducible $K_{m_0}$--conjugate families are of the form
\begin{equation}\label{eq-neigh1}
\{(t_1:1:0) \}_{m_1(t_0,t_1)}
\end{equation}
where $t_1$ is a new variable and $m_1(t_0,t_1)\in K_{m_0}[t_1]\setminus \{t_1\}$ is irreducible over $K_{m_0}$. In this situation, the method continues by applying the same process to the irreducible families of non-ordinary singularities in the first neighborhood. The irreducible families in the second neighborhood will be of the form
\begin{equation}\label{eq-neigh2}
\{(t_2:1:0) \}_{m_2(t_0,t_1,t_2)}
\end{equation}
where $t_2$ is a new variable and $m_2\in K_{m_0,m_1}[t_2]\setminus \{t_2\}$ is irreducible over $K_{m_0,m_1}$, where $K_{m_0,m_1}$ denotes the quotient field of $K_{m_0}[t_1]/\!<\!m_1(t_0,t_1)\!>$. In general, the irreducible families in the $i$-th neighborhood will be of the form
\begin{equation}\label{eq-neighs}
\{(t_i:1:0)\}_{m_i(t_0,\ldots,t_i)}.
\end{equation}
By abuse of notation, we will say that the families of the form in \eqref{eq-neigh1}, \eqref{eq-neigh2}, \eqref{eq-neighs} are $K$--families.
Applying this process till no non-ordinary neighboring conjugate family appears in the neighborhoods and until all non-ordinary families of $\Cu(G^h)$ have been blown up, we get a decomposition that we call a \textit{$K$--standard decomposition of the neighboring graph of singularities. In this situation, the genus can be computed as
\begin{equation}\label{eq-genus2}
\genus(\Cu(G))=\dfrac{(d-1)(d-2)}{2} -\dfrac{1}{2} \sum_{\mathcal{F}\in \mathscr{N}}\, \#(\mathcal{F}) \, \mathrm{mult}(\mathcal{F}) (\mathrm{mult}(\mathcal{F})-1)
\end{equation}
where $\mathscr{N}$ is a standard decomposition of the neighboring graph of $\Cu(G^h)$.}

\subsection{Parametrization algorithm}\label{subsec:param}
Once the genus of $\Cu(G^h)$ has been computed, if it is zero, we may derive a rational parametrization of the curve. There are different approaches to achieve a rational parametrization, see e.g. \cite{hoeij},  \cite{schicho}, \cite{SW91}, \cite{SW97}, \cite{myBook}. Here we will use a simplified version of the algorithm in \cite{SW97} where Hilbert--Hurwitz theorem (see e.g. Theorem. 5.8. in \cite{myBook}) is applied directly and recursively. Further improvements can be found in \cite{myBook}.

Let $\mathcal{N}$ be a $K$--standard decomposition of the neighboring graph of singularities of $\Cu(G^h)$. We consider the linear system $\mathscr{A}_{d-2}(\Cu(G^h))$ of adjoint curves to $\Cu(G^h)$ of degree $d-2$ (recall that $d=\deg(\Cu(G))$); that is,
the linear system of all $(d-2)$--degree curves having each $r$-fold point of $\Cu(G^h)$, including the neighboring ones, as a point of multiplicity at least $r-1$. In other words, $\mathscr{A}_{d-2}(\Cu(G^h))$ is the linear system of curves of degree $d-2$ defined by the   divisor
\[ \sum_{\mathcal{F}\in \mathcal{N}} \sum_{p\in \mathcal{F}} (\mathrm{mult}(p)-1)\, p, \]
where the multiplicity is w.r.t. the curve where $\mathcal{F}$ belongs to. The linear conditions to construct the linear system can be derived by working with the families in $\mathcal{N}$. In the following, we outline the process for computing $\mathscr{A}_{d-2}(\Cu(G^h))$.

\vspace*{2mm}

\noindent \underline{\textsf{Adjoints computation}}

\vspace*{1mm}

\noindent \textsf{Step 1.}
We identify the set of all projective curves, including multiple component curves, of fixed degree $d-2$, with the projective space
\begin{equation*}\label{eq-Vd}
\mathscr{V}_{d-2}:=\mathbb{P}^{\frac{(d-2)(d+1)}{2}}(\overline{K})
\end{equation*}
via their coefficients, after fixing an order of the monomials. By abuse of notation, we will refer to the elements in $\mathscr{V}_{d-2}$ by either their tuple of coefficients, or by the associated $\{x,y,z\}$--form, or by the corresponding curve. Let $H(\Lambda,x,y,z)$ denote the $\{x,y,z\}$--homogeneous polynomial of degree $d-2$, 
where $\Lambda$ is a tuple of undetermined coefficients.

\vspace*{1mm}

\noindent \textsf{Step 2.}
For each $r$-fold irreducible $K$--family $\mathcal{F}:=\{ (f_1:f_2:f_3)\}_{m_0(t_0)}\in \Stand(G^h)$, see \eqref{eq-SingularLocusApp}, and for each partial derivative $M$ of $H$ of order $r-2$ w.r.t.  $\{x,y,z\}$, compute $M(\Lambda,f_1,f_2,f_3)$ mod $m_0(t_0)$ and collect in a set $\mathcal{S}$ of its non--zero coefficients w.r.t. $t$.

\vspace*{1mm}

\noindent \textsf{Step 3.}  For each (neighboring) $r$-fold irreducible family  $\mathcal{F}:=\{(t_i:1:0)\}_{m_i(t_0,\ldots,t_i)}\in \mathcal{N}$, compute
\begin{equation}\label{eqs-neig}
\text{$W:=(\cdots (N(\Lambda,t_1,1,0)$ mod $m_i) \cdots $) mod $m_1$}
\end{equation}
where $N$ is every partial derivative of order $r-2$ of $N^*$, being $N^*$ the form obtained applying to $H$ the same blow up process (see Subsection \ref{subsec-genus}) as the one to reach the curve where the neighboring family $\mathcal{F}$ belongs to. Include in $\mathcal{S}$ all non--zero coefficients w.r.t. $t$ of $W$ \eqref{eqs-neig}.

\vspace*{1mm}

\noindent \textsf{Step 4.} Solve the homogeneous linear system of equations $\{ L(\Lambda)=0 \}_{L\in \mathcal{S}}$ and substitute the result in $H$. The  resulting form  defines $\mathscr{A}_{d-2}(\Cu(G^h))$.

\begin{remark}\label{rem-Adj}
The defining polynomial of $\mathscr{A}_{d-2}(\Cu(G^h))$ has coefficients in $K(\Lambda)$. So, the ground field is not extended.
\end{remark}

Hilbert-Hurwitz theorem (see \cite{HH} or \cite[Theorem 5.8]{myBook}) ensures that for almost all $(\phi_1,\phi_2,\phi_3)\in \mathscr{A}_{d-2}(\Cu(G^h))^3$ the mapping
\begin{equation}\label{eq-HH}
\mathscr{H}\equiv (y_1:y_2:y_3)=(\phi_1(x,y,z):\phi_2(x,y,z):\phi_3(x,y,z))
\end{equation}
transforms birationally $\Cu(G^h)$ to an irreducible curve of degree $d-2$. Note that, because of Remark~\ref{rem-Adj}, the map in \eqref{eq-HH} is defined over $K$.
Furthermore, since the map is birational, the genus is preserved. Thus, applying successively Hilbert--Hurwitz theorem one gets either a conic (if $d$ is even) or a line (if $d$ is odd) $K$--birationally equivalent to $\Cu(G^h)$.

After these considerations, we can outline the parametrization algorithm.

\vspace*{2mm}

\noindent \underline{\textsf{Parametrization computation}}


\vspace*{1mm}

\noindent \textsf{Step 1.}
Let $G$ be as above so that the genus of $\Cu(G^h)$ is zero.

\vspace*{1mm}

\noindent \textsf{Step 2.} Set $M:=G$, $\rho:= \deg(M)$, and $\mathcal{T}$ as the identity map in $\mathbb{P}^{2}(K)$.


\vspace*{1mm}

\noindent \textsf{Step 3.} While $\rho>2$ do
\begin{enumerate}
\item Compute $\mathscr{A}_{\rho-2}(\Cu(M))$.
\item Take $(\phi_1,\phi_2,\phi_3)\in \mathscr{A}_{d-2}(\Cu(G^h))^3$, so that the mapping $\mathscr{H}$ in $\eqref{eq-HH}$ is birational over $\Cu(M)$.
\item Replace $M$ by $M(\mathscr{H}^{-1})$, $\rho$ by $\deg(M)$, and $\mathcal{T}$ by $\mathscr{H} \circ \mathcal{T}$.
\end{enumerate}
\vspace*{1mm}

\noindent \textsf{Step 4.}
Parametrize birationally $\Cu(M)$. Let $\Qa(t)$ be the output parametrization.

\vspace*{1mm}

\noindent \textsf{Step 5.}
Return $\mathcal{T}^{-1}(\Qa(t))$.

\begin{remark}\
\begin{enumerate}
\item If $d$ is odd one may stop the loop in Step 3 when $\rho=3$ since cubics can be easily parametrize over the ground field. For parametrizing a conic or a cubic see \cite{myBook}.
\item The direct classical parametrization algorithm, see \cite{walker}, \cite{myBook}, for $\Cu(G^h)$, in addition to the computation of the adjoints, needs $d-2$ simple points of $\Cu(G^h)$ (indeed, a more sophisticated method in \cite{SW91} shows that only one simple point is enough). Therefore, and alternative to Step 5 is: compute $d-2$ simple points either on the conic, or on the line, for instance, using $\Qa(t)$; apply $\mathcal{T}$ to these simple points; and now use the direct parametrization algorithm  for $\Cu(G^h)$.
\end{enumerate}
\end{remark}


\noindent

\begin{thebibliography}{44}
\bibitem{arrondo1}  Arrondo, E., Sendra, J., Sendra, J.R.: Parametric Generalized Offsets to Hypersurfaces. Journal of Symbolic Computation; 23, 267–285, 1997.
\bibitem{arrondo2} Arrondo, E., Sendra, J., Sendra, J.R.: Genus Formula for Generalized Offset Curves. Journal of Pure and Aplied Algebra; 136(3):199–209, 1999.
\bibitem{mauro} Beltrametti M.C., Carletti E., Gallarati D., MontiBragadin G. Lectures on Curves, Surfaces and Projective Varieties. A classical view of algebraic geometry. EMS Textbooks (2009).
\bibitem{cox} Cox, D., Little, J., O’Shea, D. Ideals, Varieties, and Algorithms: an Introduction to Computational Algebraic Geometry and Commutative Algebra. Springer
Science \& Business Media, 14th edition, 2013.
\bibitem{tsen} Ding, S., Kang, M.-C., Tan, E.-T. Chiungtze C. Tsen (1898--1940) and Tsen's Theorems. Rocky Mountain J. Math. 29(4): 1237-1269, 1999.
\bibitem{hilbert1892} Hilbert, D., \"Uber die Irreducibilit{\"a}t ganzer rationaler Functionen mit ganzzahligen Coefficienten. Walter de Gruyter, Berlin, New York, 1892.
\bibitem{HH} Hilbert, D., Hurwitz A. \"Uber die Diophantischen Gleichungen vom
Geschlecht Null. Acta math; 14, 217--224, 1890.
\bibitem{FS}
  Falkensteiner, S., Sendra, J.R.    
Rational Solutions of Parametric First-Order Algebraic Differential Equations  
	arXiv:2307.05102 
\bibitem{FS90} Farouki, R., Sakkalis, T. Singular Points on Algebra Curves. J.  
Symb. Comp.; 9(4):405–-421, 1990.
\bibitem{Fulton} Fulton, W. Algebraic Curves-An Introduction to Algebraic Geometry. Addision-Wesley, Redwood City CA (1989).
\bibitem{Geddes} Geddes, K.O., Czapor, S. R., Labahn, G. Algorithms for Computer
Algebra. Kluwer Academic Publishers, Boston, 1992.
\bibitem{HW} Hillgarter E., Winkler F. Points on Algebraic Curves and the Parametrization Problem. Automated Deduction in Geometry. Lecture Notes in Artif. Intell. 1360: 185--203. D. Wang (ed.). Springer Verlag Berlin Heidelberg, 1998.
\bibitem{hoeij} van Hoeij M. Rational parametrization of curves using canonical divisors. J.  Symb. Comp.; 23, 209--227, 1997.
\bibitem{meromorfas} L. Kaup, B. Kaup, Holomorphic functions of several variables, An Introduction to the Fundamental Theory, Walter de Gruyter, Berlin, New York, 2011.
\bibitem{laplagne} Laplagne S. An Algorithm for the Computation of the Radical of an Ideal. ISSAC 2006 pp. 191--195. ACM Press, 2006.
\bibitem{lu}  L\"u, W.: Offset-Rational Parametric Plane Curves. Computer Aided Geometric Design; 12:601--617, 1995.
\bibitem{montes} Montes A. The Gr\"obner cover. In series: Algorithms and Computation in Mathematics. Springer--Verlag, 2018.
\bibitem{cissoids}  Peternell M., Sendra J.R., Sendra J. Cissoid Constructions of Augmented Rational Ruled Surfaces. Computer Aided Geometric Design 60:1--9, 2018.
\bibitem{PW} Poteaux A., Weimann M. Computing {Puiseux} series: a fast divide and conquer algorithm. Annales Henri Lebesgue vol 4. pp. 1061--1102, {\'ENS Rennes}, 2021.
\bibitem{schicho} Schicho, J. On the choice of pencils in the parametrization of curves. J. Symb. Comp. 14:557--576, 1992.
\bibitem{seidenberg} Seidenberg A., Constructions in algebra. Trans. Amer.
Math. Soc., (197):273--313, 1974.
\bibitem{sendraNormal} Sendra, J.R. Normal Parametrization of Algebraic Plane Curves. J. Symb. Comput. 33:863–885, 2002.
\bibitem{conchoids}  Sendra, J.R.,   Sendra. J., Rational Parametrization of Conchoids to Algebraic Curves. Applicable Algebra in Engineering, Communication and Computing, vol. 21 pp. 285--308, 2010.
\bibitem{penrose}  Sendra, J.R.,   Sendra. J., Computation of Moore-Penrose Generalized Inverses of Matrices with Meromorphic Functions. 
Applied Mathematics and Computation 313C pp. 355--366, 2017.
\bibitem{SSV} Sendra, J.R.; Sevilla, D.; Villarino, C. Some results on the surjectivity of surface parametrizations. In Lecture Notes in Computer Science 8942; Schicho, J., Weimann, M., Gutierrez, J., Eds.; Springer International Publishing:
Cham, Switzerland, pp. 192–203, 2015.
\bibitem{SW91} Sendra, J.R., Winkler, F.  Symbolic Parametrization of Curves. J. Symb. Comp. 12:607–631, 1991.
\bibitem{SW97} Sendra, J.R., Winkler, F. Parametrization of Algebraic Curves over
Optimal Field Extensions. Journal of Symbolic Computation, 23:191--207, 1997.
\bibitem{SW01} Sendra, J.R., Winkler, F. Tracing Index of Rational Curve Parametrizations.
Computer Aided Geometric Design, 18:771--795, 2001.
\bibitem{myBook} Sendra, J.R., Winkler, F., P\'erez-D\'{\i}az, S. Rational Algebraic Curves: A Computer Algebra Approach. Series: Algorithms and Computation in Mathematics. Vol. 22. Springer Verlag, 2007.
\bibitem{Serre} Serre, J.P. Baker’s Method. In: Brown, M. (eds) Lectures on the Mordell-Weil Theorem. Aspects of Mathematics, vol 15. Vieweg+Teubner Verlag, Wiesbaden, 1997.
\bibitem{shafa} Shafarevich, I.R. Basic Algebraic Geometry I and II. Springer--Verlag,
Berlin New York, 1994.
\bibitem{vo} N. Thieu Vo.
Rational and Algebraic Solutions of First-Order Algebraic ODEs. Research Institute for Symbolic Computation. PhD Thesis. 2016.
\bibitem{torrente2017} Torrente, M., Beltrametti, M.C., Sendra, J.R. Perturbation of polynomials and applications to the Hough transform. Journal of Algebra, 486:328--359, 2017.
\bibitem{wang1} Wang, D. Elimination Methods. Texts and Monographs in Symbolic Computation. Springer-Verlag, Vienna, 2001.
\bibitem{wang2} Wang, D. Elimination Practice: Software Tools and Applications.  Imperial College Press, 2004.
\bibitem{walker} Walker, R.J. Algebraic Curves. Princeton Univ. Press (1950).
\bibitem{weispfennin} Weispfenning, V. Comprehensive Gr\"obner Bases. Symbolic Computation, 14:1--29, 1992.
\bibitem{winkler} Winkler, F. Polynomial Algorithms in Computer Algebra. Springer--Verlag, Wien New York, 1996.
\end{thebibliography}
\end{document}